\documentclass[letterpaper,11pt]{amsart}

\usepackage[margin=1.2in]{geometry}
\usepackage{amsmath,amsthm,amssymb, mathtools}
\usepackage{xspace,xcolor}
\usepackage[breaklinks,colorlinks,citecolor=teal,linkcolor=teal,urlcolor=teal,pagebackref,hyperindex]{hyperref}
\usepackage[alphabetic]{amsrefs}
\usepackage{comment}
\usepackage{enumerate}
\usepackage[all]{xy}
\usepackage{tikz-cd}
\usepackage{color}
\usepackage{comment}

\theoremstyle{plain}
\newtheorem{theo}{Theorem}[section]

\newtheorem{lemm}[theo]{Lemma}
\newtheorem{prop}[theo]{Proposition}
\newtheorem{coro}[theo]{Corollary}

\theoremstyle{definition}
\newtheorem{defi}[theo]{Definition}

\newtheorem{exam}[theo]{Example}

\theoremstyle{remark}
\newtheorem{rema}[theo]{Remark}

\newcommand{\bfR}{\mathbf{R}}

\newcommand{\bfT}{\mathbf{T}}

\newcommand{\HH}{\mathbb{H}}

\newcommand{\QQ}{\mathbb{Q}}

\newcommand{\RR}{\mathbb{R}}
\newcommand{\CC}{\mathbb{C}}

\newcommand{\ZZ}{\mathbb{Z}}

\newcommand{\cH}{\mathcal{H}}

\newcommand{\cO}{\mathcal{O}}
\newcommand{\cP}{\mathcal{P}}
\newcommand{\cQ}{\mathcal{Q}}

\newcommand{\cS}{\mathcal{S}}

\newcommand{\lind}[1]{_{(#1)}}
\newcommand{\ubind}[1]{^{[#1]}}

\newcommand{\lsta}{_{*}}

\newcommand{\sta}{^{*}}

\newcommand{\dual}{^{\vee}}

\DeclareMathOperator{\Hom}{Hom}

\DeclareMathOperator{\rank}{rank}

\DeclareMathOperator{\codim}{codim}

\DeclareMathOperator{\im}{im}

\DeclareMathOperator{\id}{id}

\DeclareMathOperator{\gr}{gr}

\DeclareMathOperator{\Spec}{Spec}

\newcommand{\SheafHom}{\mathcal{H}om}

\newcommand{\SheafExt}{\mathcal{E}xt}

\newcommand{\duBois}{\underline{\Omega}}

\DeclareMathOperator{\IC}{IC}

\DeclareMathOperator{\lcd}{lcd}
\DeclareMathOperator{\lcdef}{lcdef}
\DeclareMathOperator{\Ish}{Ish}

\usepackage{soul}

\begin{document}
	\thanks{The authors were partially supported by NSF grant DMS-2301463.}
	
    
    \textbf{}
	
	\author{Hyunsuk Kim}
	
	\address{Department of Mathematics, University of Michigan, 530 Church Street, Ann Arbor, MI 48109, USA}
	
	\email{kimhysuk@umich.edu}	
	
	\author{Sridhar Venkatesh}
	
	\address{Department of Mathematics, University of Michigan, 530 Church Street, Ann Arbor, MI 48109, USA}
	
	\email{srivenk@umich.edu}
	
	\begin{abstract} 
        Given a proper toric variety and a line bundle on it, we describe the morphism on singular cohomology given by the cup product with the Chern class of that line bundle in terms of the data of the associated fan. Using that, we relate the local cohomological dimension of an affine toric variety with the Lefschetz morphism on the singular cohomology of a projective toric variety of one dimension lower. As a corollary, we show that the local cohomological defect is not a combinatorial invariant. We also produce numerous examples of toric varieties in every dimension with any possible local cohomological defect, by showing that the local cohomological defect remains unchanged under taking a pyramid.
	\end{abstract} 
	
	\title[Lefschetz and local cohomological dimension for toric varieties]{Lefschetz morphisms on singular cohomology and local cohomological dimension of toric varieties}
	
	\maketitle
	
	\setcounter{tocdepth}{1}
	
    
	\section{Introduction}


    This article is a follow-up to our paper \cite{LCDTV1} on the local cohomology and singular cohomology of toric varieties, where we study the \textit{trivial Hodge module} on toric varieties. One of the main ingredients in \textit{loc. cit.} is the fact that the Grothendieck dual of the Du Bois complex can be described very explicitly, using the \textit{Ishida complex} (see \cite{Ishida2}). In this article, we exploit this fact in more detail and give several refined results on the local cohomological dimension and singular cohomology of toric varieties. Unlike in \textit{loc. cit.}, we keep the use of Hodge modules to a minimum and instead rely on more classical techniques. Throughout the article, we work over the complex numbers $\CC$.

    The local cohomological dimension of a variety $X$ embedded inside a smooth variety $Y$ is defined as the right end of the cohomological range for the local cohomology sheaves $\cH_{X}^{q}(\cO_{Y})$, i.e.,
    $$ \lcd(Y, X) := \max \{ q : \cH_{X}^{q}(\cO_{Y}) \neq0\}.$$
    In contrast to the fact that the other end admits an easy description as 
    $$\min \{ q : \cH_{X}^{q} \cO_{Y} \neq 0\} = \codim(Y, X),$$
    the local cohomological dimension is much more subtle and interesting. An equivalent way to describe it is via the \textit{local cohomological defect} (introduced in \cite{Popa-Shen:DuBoisLCDEF}),
    $$ \lcdef(X) := \lcd(Y, X)- \codim_{Y}(X). $$ 
    In particular, we always have $\lcdef(X) \geq 0$. The local cohomological defect only depends on $X$ and not on the embedding of $X$ inside $Y$. From the description of the local cohomology in terms of the \v{C}ech complex, it is clear that $\lcd(Y, X)$ cannot exceed the number of equations locally defining $X$. In particular, if $X$ is a local complete intersection (lci) variety, then $\lcd(Y, X) = \codim_{Y}(X)$, i.e., $\lcdef(X) = 0$. Therefore, $\lcdef(X)$ can be thought of as a coarse measure of how far the variety $X$ is from being lci. For a Cohen--Macaulay variety $X$ of dimension $n$, we have $\lcdef(X) \leq \max\{0,n-3\}$ by \cite{Dao-Takagi}.

    In this article, we study the local cohomological defect of \textit{toric varieties}. A toric variety $X$ is a normal algebraic variety containing a torus $T$ as an open dense subset such that the action of $T$ on itself extends to an action of $T$ on the whole variety $X$. Toric varieties provide an interesting interplay between algebraic geometry and convex geometry since they admit an alternate description in terms of convex geometric objects. To be precise, every $n$-dimensional affine toric variety $X$ is associated to a strongly convex rational polyhedral cone $\sigma \subset N \otimes \mathbb{R}$, where $N$ is a free abelian group of rank $n$. More generally, we have a correspondence between toric varieties and \textit{fans}. See \S \ref{section:toric-var-basic} for details.

    The singularities of a toric variety are rather mild, for instance, they have rational singularities and are hence Cohen--Macaulay (in particular, the dualizing complex $\omega^\bullet_X$ is equal to the shifted dualizing sheaf $\omega_X[n]$). However, they are typically far from being lci, which makes their local cohomological defect highly interesting.

    \subsection{The local cohomological defect in terms of cones}

    The starting point of our study is the following result of Musta\c{t}\u{a}--Popa suitably restated for Cohen--Macaulay varieties. It says that the local cohomological defect of a Cohen--Macaulay variety $X$ can be calculated by analyzing the cohomologies of $\bfR\SheafHom_{\cO_{X}}(\duBois_{X}^{p}, \omega_{X})$, the Grothendieck dual of the \textit{$p$-th Du Bois complex} of $X$.

    \begin{theo}[\cite{Mustata-Popa22:Hodge-filtration-local-cohomology}*{Corollary 5.3}] \label{theo:MP-lcdef-and-depth-of-Du-Bois}
    Let $X$ be a Cohen--Macaulay variety. Then, $\lcdef(X)$ is the maximal integer $c$ satisfying the following two properties:
        \begin{enumerate}
            \item $\SheafExt_{\cO_{X}}^{j+c} (\duBois_{X}^{j}, \omega_{X}) \neq 0$ for some $j \geq 0$, and
            \item $\SheafExt_{\cO_{X}}^{j+c+1} (\duBois_{X}^{j}, \omega_{X}) = 0$ for all $j \geq 0$.
        \end{enumerate}
    \end{theo}
    Theorem \ref{theo:MP-lcdef-and-depth-of-Du-Bois} is equivalent to \cite{Mustata-Popa22:Hodge-filtration-local-cohomology}*{Corollary 5.3} by using Grothendieck duality for an embedding $X \hookrightarrow Y$ to a smooth variety $Y$, and by observing that $\omega_{X}^{\bullet} \simeq \omega_{X}[\dim X]$, since $X$ is Cohen--Macaulay.

    The $p$-th Du Bois complex $\duBois_{X}^{p}$ can be thought of as a better-behaved substitute for the K\"ahler differentials $\Omega_{X}^p$ when $X$ is singular. For instance, the Hodge-to-de Rham spectral sequence degenerates at the $E_1$-page for singular projective varieties if one uses the Du Bois complex in place of the K\"ahler differentials. For the purpose of this article, we note that the Du Bois complexes and their Grothendieck duals admit rather nice descriptions when $X$ is a toric variety. By \cite{Guillen-Navarro-Gainza:Hyperresolutions-cubiques}*{V.4}, the Du Bois complex $\duBois_{X}^{p}$ coincides with the sheaf of reflexive differentials $\Omega_{X}\ubind{p}$. On the other hand, it follows by \cite{Ishida2} that its Grothendieck dual is given by:
    $$ \bfR\SheafHom_{\cO_{X}}(\duBois_{X}^{p}, \omega_{X}) \simeq \Ish_{X}^{n-p},$$
    where $\Ish_{X}^{n-p}$ is the $(n-p)$-th \textit{Ishida complex} of $X$ (see \S\ref{section:Ishida-complex}), and $n$ is the dimension of $X$. This is a very explicit complex lying in cohomological degrees $0$ to $n-p$, whose terms consist of structure sheaves of various torus-invariant closed subsets. Moreover, if we take into account the torus action, the Ishida complex essentially decomposes into copies of $\Ish^{n-p}_\sigma$, where $\sigma$ runs over all the cones in the fan associated to $X$. Here, $\Ish^{n-p}_\sigma$ is a complex of finite dimensional vector spaces, defined purely in terms of the cone $\sigma$ (see \S\ref{section:Ishida-complex}).
    This naturally leads us to propose the following definition of the local cohomological defect of a cone.

    \begin{defi}
        For a strongly convex rational polyhedral cone $\sigma$ of dimension $n$, we define the local cohomological defect of the cone, denoted $\lcdef(\sigma)$, to be the maximal integer $c$ satisfying the following two properties:
        \begin{enumerate}
            \item $H^{j+c}(\Ish_{\sigma}^{n-j}) \neq 0$ for some $j \geq 0$, and
            \item $H^{j+c+1}(\Ish_{\sigma}^{n-j}) = 0$ for all $j \geq 0$,
        \end{enumerate}
        where $H^k(\Ish^l_\sigma)$ denotes the $k$-th cohomology group of the complex $\Ish^l_\sigma$.
    \end{defi}

    The key point is that $\Ish^{l}_{\sigma}$ is a complex of finite dimensional vector spaces constructed in a very explicit manner from $\sigma$ and hence, computing $\lcdef(\sigma)$ can be done algorithmically. From the local cohomological defect of the cone, it is easy to show that we can recover the local cohomological defect of the toric variety $X$.

    \begin{prop}\label{prop:lcdef-variety-in-terms-of-cone}
        Let $X$ be an affine toric variety corresponding to a strictly convex rational polyhedral cone $\sigma$. Then we have
        $$ \lcdef(X)= \max_{\tau \subset \sigma} \lcdef(\tau)$$
        where $\tau$ runs through the faces of $\sigma$.
    \end{prop}


    In particular, this answers a question of Musta\c{t}\u{a}-Popa \cite{Mustata-Popa22:Hodge-filtration-local-cohomology}*{Remark 4.31} in the sense that one can write a computer program computing the local cohomological dimension of toric varieties. Yet, a more direct relation between the cone and the vanishing and non-vanishing behavior of the Ishida complexes needs further investigation, and this seems to be a very interesting and subtle problem as we already see in the 4-dimensional case (see Examples \ref{exam:lcdef-non-combinatorial} and \ref{exam:our-method-doesnot-work}).

    \subsection{Lefschetz morphism on the singular cohomology}
    We now state our main result which relates the cohomologies of $\Ish^l_\sigma$ with the Lefschetz morphism on singular cohomology of a projective toric variety of one dimension lower. This is philosophically similar to the `global-to-local' principle appearing in \cite{Fieseler-ICprojtoric}.

    \begin{theo} \label{theo:lefschetz-ishida-correct-ver}
        Let $X$ be an $n$-dimensional affine toric variety associated to a full-dimensional cone $\sigma$. Let $\rho$ be a rational ray in the interior of $\sigma$ and consider the toric morphism $\pi \colon \widetilde{X} \to X$ corresponding to inserting the ray $\rho$ in $\sigma$. Let $E$ be the projective toric variety given by the inverse image of the torus fixed point. Then we have the long exact sequence
        $$ \ldots \to  \HH^{i-1}(E, \Ish_{E}^{l-1}) \xrightarrow{c_{1}\dual} \HH^{i}(E, \Ish_{E}^{l}) \to H^{i}(\Ish_{\sigma}^{l}) \to \HH^{i}(E, \Ish_{E}^{l-1}) \xrightarrow{c_{1}\dual} \ldots $$
        where $c_{1}\dual \colon \HH^{i-1}(E, \Ish_{E}^{l-1}) \to \HH^{i}(E,\Ish_{E}^{l})$ is dual to $c_1:H^{n-i-1}(E,\duBois^{n-l-1}_E) \to H^{n-i}(E, \duBois^{n-l}_E)$, the Chern class map of the $\QQ$-Cartier divisor class $(-E)|_{E}$ on $E$.
    \end{theo}

    In \cite{LCDTV1}*{Theorem 1.3}, we prove that the last map of $\Ish^{p}_X$ is surjective for $p \geq n/2$. In particular,
    we have that $H^{p}(\Ish^p_\sigma) = 0$ for $p \geq n/2$. We can recover that result using Theorem \ref{theo:lefschetz-ishida-correct-ver}, along with the following Hard Lefschetz type injectivity result.

    \begin{prop}\label{prop:hard-lefschetz-injectivity}
        Let $X$ be an $n$-dimensional projective toric variety and $L$ be an ample $\QQ$-divisor on $X$. Then for $p \leq \frac{n-1}{2}$, the morphism
        $$ c_{1}(L) \colon H^{p}(X, \duBois_{X}^{p}) \to H^{p+1}(X, \duBois_{X}^{p+1})$$
        is injective.
    \end{prop}
    
    To illustrate the utility of Theorem \ref{theo:lefschetz-ishida-correct-ver}, we give a simple characterization of the local cohomological defect for 4-dimensional toric varieties. By \cite{Dao-Takagi}, the local cohomological defect can either be $0$ or $1$ in this case.

    \begin{coro} \label{coro:lcdef-of-4dim}
        Let $X$ be a 4-dimensional affine toric variety associated to a full-dimensional cone $\sigma$. Consider a rational ray $\rho$ in the interior of $\sigma$. Consider the associated toric morphism $\pi \colon \widetilde{X} \to X$ and let $E$ be the inverse image of the torus fixed point. Then $\lcdef(X) = 1$ if and only if $\dim H^{2}(E,\mathbb{C}) \geq 2$.
    \end{coro}

    As a consequence of Corollary \ref{coro:lcdef-of-4dim}, we immediately get the following example which shows that the local cohomological defect of an affine toric variety is not a combinatorial invariant of the associated cone.
	\begin{exam} \label{exam:lcdef-non-combinatorial}
		Let $\sigma$ be the convex cone in $\RR^{4}$ generated by the following 14 rays:
		\begin{align*}
			&(1,0,0,1), (-1,0,0,1), (0,-1,0,1), (0,1,0,1), (0,0,1,1),\\
			& (0,0,-1,1), (1,1,1,2) ,(-1,1,1,2), (1,-1,1,2), (-1,-1,1,2),\\
			& (1,1,-1,2), (-1,1,-1,2), (1,-1,-1,2), (-1,-1,-1,2).
		\end{align*}
		Let $\sigma'$ be the convex cone in $\RR^{4}$ generated by the following 14 rays:
		\begin{align*}
			& (1,0,0,1) , (0,1,0,1), (-1,0,1,2), (-1,0,0,1), (0,-1,0,1),\\ 
			& (0,0,-1,1), (2,3,1,5), (1,1,-1,2), (2,-3,1,5), (1,-1,-1,2),\\
			& (-2,1,1,3), (-1,1,-1,2), (-2,-1,1,3), (-1,-1,-1,2)
		\end{align*}
		Let $X$ and $X'$ be the toric varieties corresponding to $\sigma$ and $\sigma'$ respectively. The two cones have the same combinatorial data, however $\lcdef(X) = 1$ and $\lcdef(X') = 0$. While one can directly check this using Macaulay2, Corollary \ref{coro:lcdef-of-4dim} provides a more conceptual reason since this example essentially comes from \cite{CoxLittleSchenck-ToricVar}*{Exercise 12.3.11}, which consists of two projective toric threefolds $E$ and $E'$ having the same combinatorial data, but $\dim H^2(E,\mathbb{C}) = 2$ while $\dim H^2(E',\mathbb{C}) = 1$.
	\end{exam}

    \subsection{Some combinatorial results about the local cohomological defect}
    Even though Example \ref{exam:lcdef-non-combinatorial} establishes that the local cohomological defect is not a combinatorial invariant, in this subsection we state a few results regarding the local cohomological defect which have a more combinatorial nature. More specifically, we prove that various classes of affine toric varieties which come from combinatorially similar cones have the same local cohomological defect. This suggests that even though the $\lcdef$ is not a combinatorial invariant, there is still plenty of scope to study it using combinatorics.

    The following theorem allows us to give a huge class of examples of $n$-dimensional toric varieties with local cohomological defect from $0$ to $n-3$. Roughly speaking, it states that the local cohomological defect is unchanged if we take a `pyramid', i.e., if we add a new ray in a linearly independent direction. 
	
	\begin{theo} \label{theo:adding-new-ray}
		Let $\sigma \subset N \otimes \mathbb{R}$ be a full dimensional cone of dimension $n$. Let $\widetilde{N} = N \oplus \ZZ e_{n+1}$ and let $\rho$ be a ray in $\widetilde{N}$ which is not contained in $N$. Let $\prescript{\rho}{}{\sigma}:= \mathrm{span}_{\RR\geq 0}(\rho, \sigma) \subset \widetilde{N}$. Let $X$ and $\widetilde{X}$ be the affine toric varieties associated to $\sigma$ and $\widetilde{\sigma}$ respectively. Then we have $\lcdef(X) =\lcdef(\widetilde{X})$.
	\end{theo}
    Given an $m$ with $0 \leq m \leq n-3$, if we start with an $(m+3)$-dimensional non-simplicial toric variety with isolated non-simplicial locus, then \cite[Theorem 1.4]{LCDTV1} tells us that its lcdef is equal to $m$. Then by a repeated application of Theorem \ref{theo:adding-new-ray}, one can obtain a huge class of $n$-dimensional toric varieties with local cohomological defect equal to $m$.


    We end the subsection by stating some results specific to dimension 4. First, we have the following simple observation.

    \begin{prop} \label{prop:Euler-char-lcdef}
        Let $\sigma$ be a 4-dimensional full dimensional cone whose number of 3-dimensional faces is strictly larger than the number of 1-dimensional faces. Then $\lcdef(X) = 1$ where $X$ is the affine toric variety associated to $\sigma$.
    \end{prop}

    Next, we use the notion of shelling of a cone to prove the following results.

    \begin{theo}\label{theo:shelling_contains_vertex}
        Let $\sigma$ be a 4-dimensional full-dimensional cone which admits a shelling order $f_{1},\ldots, f_{r}$ such that for $i \neq r-1, r$, $$ f_{i} \setminus \bigcup_{j < i} f_{j}$$ contains a ray. Then $\lcdef(X) = 0$ where $X$ is the associated affine toric variety.
    \end{theo}

    \begin{theo}\label{theo:lcdef-of-not-a-pyramid-dim-4}
    Let $\sigma \subset N \otimes \mathbb{R}$ be a 4-dimensional cone which has a ray $\tau_0$ such that every facet (i.e. codimension $1$ face of $\sigma$) containing $\tau_0$ is simplicial. Assume additionally that the vector space spanned by all the other rays $\{\tau \subset\sigma \mid \dim(\tau) = 1,  \tau \neq \tau_0 \}$ is $N \otimes \mathbb{R}$. Then $\lcdef(X) = 1$ where $X$ is the associated affine toric variety.
    \end{theo}

    The span condition in Theorem \ref{theo:lcdef-of-not-a-pyramid-dim-4} is included to rule out examples coming from Theorem \ref{theo:adding-new-ray}. We use this theorem to produce interesting 4-dimensional examples of varieties with $\lcdef = 1$ (see Example \ref{exam:dor}).

    \begin{rema}
		We remark that the local cohomological dimension is interesting only in characteristic zero, and its behavior is drastically different in positive characteristics. Indeed, toric varieties are Cohen--Macaulay, hence if we embed a toric variety $X$ into a smooth variety $Z$, we have $\lcd(Z, X) - \codim_{Z}X=0$ by \cite{Peskine-Szpiro:char-p-lcdef}*{\S III. Proposition 4.1}. In that article, the authors exploit the action of Frobenius on local cohomology.
	\end{rema}

    \subsection{Organization of the paper}

    We discuss some preliminaries on the Du Bois complex and toric varieties in \S \ref{sec:preliminaries}. We then relate the Lefschetz morphism on singular cohomology with the local cohomological defect and prove Theorem \ref{theo:lefschetz-ishida-correct-ver} and its consequences in \S \ref{section:Lefschetz}. Finally, in \S \ref{sec:combinatorial-results-lcdef}, we prove the combinatorially flavored results, namely Theorems \ref{theo:adding-new-ray}, \ref{theo:shelling_contains_vertex}, \ref{theo:lcdef-of-not-a-pyramid-dim-4} and Proposition \ref{prop:Euler-char-lcdef}.

	\section{Preliminaries}\label{sec:preliminaries}
    \subsection{The Du Bois complex}
    In \cite{DuBois:complexe-de-deRham}, Du Bois introduced a filtered complex $\duBois_{X}^{\bullet}$ which can be thought of as a correct replacement of the de Rham complex $\Omega_{X}^{\bullet}$ when $X$ is singular. By taking the graded quotients, the {\it $p$-th Du Bois complex} is defined as
	$$ \duBois_{X}^{p} := \gr_{F}^{p} \duBois_{X}^{\bullet} [p] \in D^b_{\rm coh}(X).$$
    We have a natural comparison map $\Omega_{X}^{\bullet} \to \duBois_{X}^{\bullet}$ of filtered complexes which is an isomorphism if $X$ is smooth, where the filtration on $\Omega_{X}^{\bullet}$ is given by the stupid filtration.

    The Du Bois complex is indeed the `correct' object to consider when $X$ is singular. For example, we have an isomorphism $\CC_{X} \simeq \duBois_{X}^{\bullet}$ and so, the hypercohomology of $\duBois_{X}^{\bullet}$ computes the singular cohomology of $X$. If $X$ is a proper variety, then the spectral sequence computing the singular cohomology degenerates at $E_{1}$. In particular, the filtration given by the spectral sequence agrees with Deligne's Hodge filtration on the singular cohomology of algebraic varieties in the following sense:
    $$ F^{p}H^{k}(X, \CC) = \im (\HH^{k}(X, F^{\geq p}\duBois_{X}^{\bullet}) \to H^{k}(X, \CC)).$$
    Hence, the graded quotient can be expressed as $\gr_{F}^{p} H^{k}(X,\CC) \simeq \HH^{k-p}(X, \duBois_{X}^{p})$.

    By \cite{Guillen-Navarro-Gainza:Hyperresolutions-cubiques}*{V.4}, the Du Bois complex $\duBois_{X}^{p}$ coincides with the sheaf of reflexive differentials $\Omega_{X}\ubind{p}$ when $X$ is a toric variety. In particular, $\duBois_{X}^{p}$ is a sheaf in this case. Since all of the varieties that we deal with are toric, we will not distinguish the Du Bois complex and the sheaf of reflexive differentials throughout this article.
    
	\subsection{Toric varieties} \label{section:toric-var-basic}
	We follow \cite{Fulton-ToricVar, CoxLittleSchenck-ToricVar} for general notions of toric varieties. To a strongly convex rational polyhedral cone $\sigma$, we associate an affine toric variety $X_{\sigma} = \Spec \CC [\sigma\dual \cap M]$. In general, to a fan $\cP$, we associate a toric variety by gluing the affine toric varieties corresponding to the cones of $\cP$.
	
	Before going into toric varieties, we set up some notation for convex cones. From now on, all cones are strongly convex rational polyhedral. Let $N$ be a free abelian group of rank $n$ and let $M:= \Hom_{\ZZ}(N, \ZZ)$. Denote $N_{\RR} := N \otimes_{\ZZ} \RR$ and $M_{\RR} := M \otimes_{\ZZ} \RR$. Let $\sigma$ be a cone in $N_{\RR}$. We denote $\cP$ the collection of all faces of $\sigma$ and view $(\cP, \subseteq)$ as a graded poset. For an integer $m \in [0, n]$, we denote by
	$$ \cP_{m} = \{ \lambda \in \cP : \dim \tau = m \}. $$
	For $\mu \in \cP$, we set
	$$ \cP^{\subset \mu} := \{\lambda \in \cP : \lambda \subset \mu\}, \quad \cP^{\supset \mu} := \{ \lambda \in \cP: \lambda \supset \mu\}.$$
	Also, we set $\cP_{m}^{\subset \mu}:= \cP_{m} \cap \cP^{\subset \mu} $ and $\cP_{m}^{\supset \mu}:= \cP_{m} \cap \cP^{\supset \mu}$.
	
	Let $\mu, \tau \in \cP$.
	\begin{enumerate}
		\item $\tau\dual := \{ u \in M_{\RR} : \langle u , v \rangle \geq 0 \text{ for all }v \in \tau \} $
		\item $\tau^{\perp} := \{ u \in M_{\RR} : \langle u , v \rangle = 0 \text{ for all }v \in \tau \} $
		\item $\tau\sta_{\circ} := (\tau^{\perp} \cap \sigma\dual \cap M) \setminus \bigcup_{\tau \subsetneq \nu} (\nu^{\perp} \cap \sigma\dual \cap M)$
		\item $\langle \tau \rangle \subset N_{\RR}$ is the subspace spanned by $\tau$
		\item $d_{\tau} = \dim_{\RR} \langle\tau\rangle$
		\item We say $\sigma$ is full-dimensional if $d_{\sigma} = \rank_{\ZZ} N$.
		\item $\tau$ is \textit{simplicial} if the 1-dimensional faces (i.e. rays) of $\tau$ are linearly independent over $\RR$ in $N_{\RR}$.
	\end{enumerate}
	
	\begin{rema} \label{rema:grading-and-face-relations}
		We remark that there is an order-reversing one-to-one correspondence between the faces of $\sigma$ and the faces of $\sigma\dual$ by sending $\tau$ to $\tau^{\perp} \cap \sigma^{\dual}$. We also point out that $\{\tau_{\circ}\sta | \tau \in \cP \}$ gives a partition of the set $\sigma\dual \cap M$. It is straightforward to check that for $u \in \tau_{\circ}\sta$, we have
		$$ u \in \mu^{\perp} \cap \sigma\dual \cap M \quad \text{if and only if} \quad \mu \subset \tau.$$
	\end{rema}
	
	\noindent
	We briefly describe the structure of affine charts, torus-invariant closed subsets and the orbits, following \cite{Fulton-ToricVar}*{Section 3.1}. Let $X =\Spec \CC[\sigma\dual \cap M]$ be the affine toric variety associated to a cone $\sigma$. For an $r$-dimensional face $\tau$ of $\sigma$, we get an irreducible torus-invariant subvariety $S_{\tau}$ of codimension $r$ given by $ \Spec \CC[\sigma\dual \cap \tau^{\perp} \cap M]$. This is the affine toric variety corresponding to the cone $\overline{\sigma}_{\tau}$, where $\overline{\sigma}_{\tau}$ is the image of $\sigma$ under the projection map $N_{\RR} \to N_{\RR} / \langle \tau \rangle$. The lattice and the dual lattice of $S_{\tau}$ is given by
	$$ N_{\tau} := \frac{N}{N \cap \langle \tau \rangle}, \qquad M_{\tau} := M \cap \tau^{\perp}.$$
	We denote by $O_{\tau} = \Spec \CC[M_{\tau}]$ the torus orbit corresponding to $\tau$, and $U_{\tau} =\Spec \CC[\tau^{\dual} \cap M]$ the affine chart of $X$ corresponding to $\tau$. We have a diagram of torus equivariant morphisms
	$$ \begin{tikzcd}
		U_{\tau} \ar[r, hook] \ar[d]& X_{\sigma} \ar[d] \\ O_{\tau} \ar[r, hook] & S_\tau.
	\end{tikzcd}$$
	Here, the horizontal arrows are open immersions. Also, after fixing a non-canonical splitting $N = N_{\tau} \oplus (N \cap \langle\tau\rangle)$ and the corresponding splitting $M = M_{\tau} \oplus M'$, we can identify the vertical map $U_{\tau} \to O_{\tau}$ as the projection $U_{\tau} = V_{\tau} \times O_{\tau} \to O_{\tau}$, where $V_{\tau}$ is the full-dimensional toric variety $\Spec \CC[\tau\dual \cap M']$, by viewing $\tau$ as a cone in $\langle \tau \rangle \cap N_{\RR}$.
	
	We also mention that by \cite{CoxLittleSchenck-ToricVar}*{Theorem 9.2.5}, toric varieties are normal and they have rational singularities, hence are Cohen--Macaulay.

        \subsection{Differential forms on toric varieties} \label{section:diff-forms-on-toric}
	We briefly discuss how differential forms work on toric varieties. Note that $M$ can be identified with the group of characters of the torus $\Hom(\bfT, \CC^{\times})$, where $\bfT$ is the torus. Hence, for $u \in M$, one can associate a differential 1-form $d \log \chi^{u} := \chi^{-u} \cdot d\chi^{u}$ on $\bfT$, where $\chi^{u} : \bfT \to \CC^{\times}$ is the character corresponding to $u$. One can show that these differential forms extend as logarithmic differential forms on the whole space $X$ (in a suitable sense), and give an isomorphism
	$$ M \otimes_{\ZZ} \cO_{X} \simeq \Omega_{X}\ubind{1} (\log D),$$
	where $D$ is the sum of the torus-invariant divisors. Here, $\Omega_{X}\ubind{1} (\log D) := j\lsta \Omega_{X \setminus Z}^{1} (\log D|_{X \setminus Z})$ where $Z$ is the union of codimension 2 torus-invariant subspaces, so $X \setminus Z$ is smooth and $D|_{X \setminus Z}$ is a smooth divisor, and $j : X \setminus Z \to X$ is the open inclusion. From this, we can see that
	$$ \bigwedge^{l} M \otimes_{\ZZ} \cO_{X} \simeq \Omega_{X}\ubind{l} (\log D),$$
	where $\Omega_{X}\ubind{l}(\log D)$ is defined analogously.

    Consider an irreducible torus-invariant divisor $S_{\rho}$ on $X$ where $\rho$ is the corresponding ray. A logarithmic form $\alpha \in \bigwedge^{l} M \otimes \cO_{X}$ is a differential form (i.e., an element of $\Omega_{X}^{l}$) on a neighborhood of the torus orbit $O_{\rho}$ if and only if $\alpha$ lies in the kernel of
    $$ \bigwedge^{l} M \otimes \cO_{X} \to \bigwedge^{l-1} \rho^{\perp} \otimes \cO_{S_{\rho}}.$$
    The map $\bigwedge^{l} M \to \bigwedge^{l-1} \rho^{\perp}$ is given by the contraction with the primitive element of $\rho$ (see \S\ref{section:Ishida-complex} for the description in terms of the Ishida complex). In particular, this says that for a collection of torus-invariant divisors $E = \sum_{\rho \in I} S_{\rho}$, we have
    $$ \Omega_{X}\ubind{l}(\log E) = \ker \left( \bigwedge^{l} M \otimes \cO_{X} \to \bigoplus_{\mu \notin I} \bigwedge^{l-1} \mu^{\perp} \otimes \cO_{S_{\mu}} \right),$$
    where the sum on the right runs over all torus-invariant divisors $S_{\mu}$ on $X$ for $\mu$ not contained in $I$. This can be seen following the lines of \cite{LCDTV1}*{Proposition 4.7}.
	
	\subsection{Shelling} \label{section:shelling}
	We introduce the concept of \textit{shelling}. While the shelling is usually considered for polytopes, we use the language of cones, since it is better for our purposes.
	
	\begin{defi}
		Let $\sigma$ be a cone of dimension $n$. Let $\cP$ be the fan associated to $\sigma$, which is the collection of all faces of $\sigma$. A \textit{shelling} of $\sigma$ is a linear ordering $\mu_{1},\ldots, \mu_{s}$ of $\cP_{n-1}$ such that either $n = 1$, or it satisfies the following condition:
		\begin{enumerate}
			\item The set of facets $\cP_{n-2}^{\subset \mu_{1}}$ of the first facet $\mu_{1}$ has a shelling.
			\item For $1 < j \leq s$,
			$$ \mu_{j} \cap \left( \bigcup_{i=1}^{j-1} \mu_{i}\right) = \lambda_1 \cup \ldots \cup \lambda_r $$
			for some shelling $\lambda_1,\ldots, \lambda_r, \ldots, \lambda_t$ of $\cP_{n-2}^{\subset \mu_{j}}$.
		\end{enumerate}
		We say a cone is \textit{shellable} if it admits a shelling.
	\end{defi}
	
	By \cite{Bruggesser-Mani:Shellable}, all cones are shellable. Indeed, the shelling of a polytope of dimension $n-1$ obtained by a suitable hyperplane section of the cone provides a shelling of the cone itself.

    \subsection{Ishida complex} \label{section:Ishida-complex}
    In this section, we recall some basic definitions regarding the Ishida complex \cite{Ishida2} and prove Proposition \ref{prop:lcdef-variety-in-terms-of-cone}. We refer to \cite{Ishida2} or \cite{LCDTV1}*{\S4} for more details and proofs. We fix a toric variety $X$ associated to a fan $\cP$ in $N$. For $\mu \subset \tau$ faces of $\sigma$ with $d_{\tau} = d_{\mu} + 1$, we denote by $n_{\mu, \tau}$ an element in $N$ such that $\langle \cdot, n_{\mu,\tau} \rangle : M \to \ZZ$ is zero on $\tau^{\perp}\cap M$ and maps $\tau\dual \cap \mu^{\perp}\cap M$ onto $\ZZ_{\geq 0}$. Note that this element is well-defined modulo $\langle\mu\rangle \cap N$. Then we define the $l$-th Ishida complex as
    $$ \Ish_{X}^{l} : \bigwedge^{l} M_{\RR} \otimes_{\RR} \cO_{X} \to \bigoplus_{\mu \in\cP_{1}} \bigwedge^{l-1} \mu^{\perp} \otimes_{\RR} \cO_{S_{\mu}} \to \ldots \to \bigoplus_{\mu \in \cP_{l}} \RR_{\mu} \otimes_{\RR} \cO_{S_{\mu}}. $$
    This complex lives in cohomological degrees 0 to $l$. The maps in the complex are given as follows. If $\mu \in \cP_{m}$ and $\tau \in \cP_{m+1}$ with $\mu \subset \tau$, then we have a morphism $\varphi_{\mu, \tau}^{l} \colon\bigwedge^{l-m} \mu^{\perp} \to \bigwedge^{l-m-1} \tau^{\perp}$ given by the contraction by $n_{\mu, \tau}$. The corresponding map in the complex is given by $\varphi_{\mu, \tau}^{l}$ tensored with the restriction morphism $\cO_{S_{\mu}} \to \cO_{S_{\tau}}$. The fact that this is indeed a complex directly translates to the following easy linear algebra fact:
    \begin{lemm} \label{lemm:anti-commutativity-complex}
        Let $\mu \in \cP_{m}$ and $\tau \in \cP_{m+2}$ with $\mu \subset \tau$. Then there exist exactly two elements $\lambda_1$ and $\lambda_2$ in $\cP_{m+1}$ such that $\mu \subset \lambda_i \subset \tau$. Furthermore, we have
		$$ \varphi_{\lambda_1, \tau}^{l} \circ \varphi_{\mu,\lambda_1}^l + \varphi_{\lambda_2, \tau}^{l} \circ \varphi_{\mu,\lambda_2}^{l} = 0. $$
    \end{lemm}

    We similarly define the complex of finite dimensional vector spaces $\Ish^l_{\cP}$ as follows:
    $$ \Ish_{\cP}^{l} : \bigwedge^{l} M_{\RR}  \to \bigoplus_{\mu \in\cP_{1}} \bigwedge^{l-1} \mu^{\perp}  \to \ldots \to \bigoplus_{\mu \in \cP_{l}} \RR_{\mu}. $$

	If $X$ is an affine toric variety corresponding to a full-dimensional cone $\sigma \subset N$, then the complex $\Ish_{X}^{l}$ carries a natural grading by the group of characters $M$, and one can easily see that $\Ish_{\sigma}^{l}$ is exactly the degree zero part of $\Ish_{X}^{l}$, with respect to this grading. Here, we consider $\sigma$ also as the fan given by the collection of all faces of $\sigma$.
    
    The Ishida complex agrees with the Grothendieck dual of the Du Bois complex.
	\begin{prop} \label{prop:Gro-dual-of-Ishida-is-duBois}\cite{Ishida2}
	    Let $X$ be a toric variety. Then
        $$ \Ish_{X}^{l} \simeq \bfR \SheafHom_{\cO_{X}} (\duBois_{X}^{n-l} ,\omega_{X}).$$
	\end{prop}
    It is easy to describe other graded pieces of the Ishida complex using the notation above.
	\begin{lemm} \label{lemm:grade-parts-of-Ishida-complex}
		Let $u \in \tau_{\circ}\sta$ for some $\tau \in \cP$. Then the degree $u$-part of the Ishida complex $\Ish_{X}^{l}$ is isomorphic to
		$$ \bigoplus_{j=0}^{l} \bigwedge^{j} \tau^{\perp} \otimes \Ish_{\tau}^{l-j}  $$
		with the convention that $\Ish_{\tau}^{j}$ and $\bigwedge^{j} \tau^{\perp}$ is zero if $j > d_{\tau}$.
	\end{lemm}
	
	We now simply rephrase Theorem \ref{theo:MP-lcdef-and-depth-of-Du-Bois} in terms of the Ishida complex.
	\begin{prop} \label{prop:lcd-intermsof-Ishida}
		Let $X$ be a toric variety of dimension $n$. Then $\lcdef(X)$ is the maximal integer $c$ satisfying the following two properties:
        \begin{enumerate}
            \item $H^{j+c}(\Ish_{X}^{n-j}) \neq 0$ for some $j \geq 0$, and
            \item $H^{j+c+1}(\Ish_{X}^{n-j}) = 0$ for all $j \geq 0$.
        \end{enumerate}
	\end{prop}
    Combining Lemma \ref{lemm:grade-parts-of-Ishida-complex} and Proposition \ref{prop:lcd-intermsof-Ishida} immediately gives a proof of Proposition \ref{prop:lcdef-variety-in-terms-of-cone}.

	\subsection{Some linear algebra lemmas}
	Here, we provide two small linear algebra lemmas that we will use later. We consider $\mu \subset \tau$ two faces of $\sigma$ with $d_{\tau} = d_{\mu} + 1$, and $\rho$ a ray not contained in $\tau$. Consider $\prescript{\rho}{}{\mu} = \mathrm{span}_{\RR_{\geq 0}}(\mu, \rho)$ and $\prescript{\rho}{}{\tau}$ defined analogously. Suppose that $\mu, \tau, \prescript{\rho}{}{\mu}, \prescript{\rho}{}{\tau}$ are all faces of $\sigma$. Note that $n_{0, \rho}$ is the primitive element in $\rho$.
	
	\begin{defi} \label{defi:a_mu}
		We define
		$$ a_{\mu} := \# \frac{\langle \prescript{\rho}{}{\mu} \rangle \cap N}{ \ZZ \cdot n_{0, \rho} + \langle \mu \rangle \cap N}.$$
		This is a positive integer, since $\ZZ \cdot n_{0, \rho} + \langle \mu \rangle \cap N$ is a finite index subgroup of $\langle \prescript{\rho}{}{\mu}\rangle \cap N$. 
	\end{defi}
	
	\begin{lemm} \label{lemm:vector_multiply_by_amu}
		We have $n_{0, \rho} = a_{\mu} n_{\mu, \prescript{\rho}{}{\mu}}$ modulo $\langle\mu\rangle \cap N$.
	\end{lemm}
	\begin{proof}
		The group $\langle \prescript{\rho}{}{\mu} \rangle \cap N / \langle \mu \rangle \cap N$ is torsion-free of rank 1, and $n_{0, \rho}$ is a non-trivial element in this group. $a_{\mu}$ is exactly the divisibility of $n_{0, \rho}$ in this group. Also, $n_{\mu, \prescript{\rho}{}{\mu}}$ modulo $\langle \mu \rangle \cap N$ is a generator of this group lying in $\prescript{\rho}{}{\mu}$. Hence, we have $n_{0, \rho} = a_{\mu} a_{\mu, \prescript{\rho}{}{\mu}}$ modulo $\langle \mu \rangle \cap N$.
	\end{proof}
	
	\begin{lemm} \label{lemm:vector_multiply_amu2}
		In the above set-up, we have $a_{\mu} n_{\mu, \tau} = a_{\tau} n_{\prescript{\rho}{}{\mu}, \prescript{\rho}{}{\tau}}$ modulo $\langle \prescript{\rho}{}{\mu}\rangle \cap N$. Hence, the following diagram commutes:
		$$ \begin{tikzcd}
			\bigwedge^{l} \prescript{\rho}{}{\mu}^{\perp} \ar[r, "\varphi_{\prescript{\rho}{}{\mu}, \prescript{\rho}{}{\tau}}"] \ar[d, "a_{\mu}"] & \bigwedge^{l-1} \prescript{\rho}{}{\tau}^{\perp} \ar[d, "a_{\tau}"] \\
			\bigwedge^{l} \mu^{\perp} \ar[r, "\varphi_{\mu, \tau}"] & \bigwedge^{l-1} \tau^{\perp}.
		\end{tikzcd} $$
		The left vertical arrow is induced by the inclusion $\prescript{\rho}{}{\mu}^{\perp} \to \mu^{\perp}$ and multiplication by $a_{\mu}$. The right arrow is defined analogously.
	\end{lemm}
	\begin{proof}
		Note that $\prescript{\rho}{}{\mu}$ and $\tau$ are the two faces of $\prescript{\rho}{}{\tau}$ containing $\mu$. Let $\alpha \in \prescript{\rho}{}{\mu}^{\perp}$ and $\beta \in \tau^{\perp}$ not contained in $\prescript{\rho}{}{\tau}^{\perp}$. These vectors are uniquely determined up to a scaling and modulo $\prescript{\rho}{}{\tau}^{\perp}$. We have
		$$ \varphi_{\prescript{\rho}{}{\mu}, \prescript{\rho}{}{\tau}}\circ \varphi_{\mu, \prescript{\rho}{}{\mu}} (\alpha \wedge \beta) = \varphi_{\prescript{\rho}{}{\mu}, \prescript{\rho}{}{\tau}} (- \langle n_{\mu, \prescript{\rho}{}{\mu}}, \beta \rangle \alpha) = -\langle n_{\mu, \prescript{\rho}{}{\mu}}, \beta \rangle \langle n_{\prescript{\rho}{}{\mu}, \prescript{\rho}{}{\tau}} , \alpha\rangle. $$
		Similarly, we have
		$$ \varphi_{\tau, \prescript{\rho}{}{\tau}}\circ \varphi_{\mu, \tau} (\alpha \wedge \beta) = \varphi_{\tau, \prescript{\rho}{}{\tau}} (\langle n_{\mu, \tau},\alpha \rangle \beta) = \langle n_{\mu, \tau},\alpha \rangle \langle n_{\tau, \prescript{\rho}{}{\tau}}, \beta \rangle . $$
		Since $\varphi_{\prescript{\rho}{}{\mu}, \prescript{\rho}{}{\tau}}\circ \varphi_{\mu, \prescript{\rho}{}{\mu}}  + \varphi_{\tau, \prescript{\rho}{}{\tau}}\circ \varphi_{\mu, \tau} = 0$ by Lemma \ref{lemm:anti-commutativity-complex}, we have
		$$\langle n_{\mu, \prescript{\rho}{}{\mu}}, \beta \rangle \langle n_{\prescript{\rho}{}{\mu}, \prescript{\rho}{}{\tau}} , \alpha\rangle = \langle n_{\mu, \tau},\alpha \rangle \langle n_{\tau, \prescript{\rho}{}{\tau}}, \beta \rangle. $$
		By Lemma \ref{lemm:vector_multiply_by_amu}, we have $a_{\mu} n_{\mu, \prescript{\rho}{}{\mu}} = a_{\tau} n_{\tau, \prescript{\rho}{}{\tau}}$ modulo $\langle \tau \rangle \cap N$. This implies $\langle a_{\mu}n_{\mu, \prescript{\rho}{}{\mu}}, \beta \rangle = \langle a_{\tau}n_{\tau, \prescript{\rho}{}{\tau}}, \beta \rangle$. Therefore, we get
		$$ \langle a_{\tau} n_{\prescript{\rho}{}{\mu}, \prescript{\rho}{}{\tau}}, \alpha \rangle = \langle a_{\mu}n_{\mu, \tau}, \alpha\rangle.$$
		This is equivalent to $a_{\mu}n_{\mu, \tau} = a_{\tau} n_{\prescript{\rho}{}{\mu}, \prescript{\rho}{}{\tau}}$ modulo $\langle \prescript{\rho}{}{\mu}\rangle \cap N$, as well as the commutativity of the diagram that we want.
	\end{proof}

    \section{Singular cohomology and the Lefschetz morphism} \label{section:Lefschetz}
    In this section, we consider the Lefschetz morphisms on the singular cohomology of proper toric varieties and relate them to the local cohomological defect. Let $X$ be a proper toric variety of dimension $n$. Let $\cP$ be the corresponding fan. In \cite{CoxLittleSchenck-ToricVar}*{\S12.3}, one uses the spectral sequence associated to a filtered topological space in order to compute the singular cohomology groups $H^{k}(X, \ZZ)$. This spectral sequence degenerates at $E_{2}$ and an Ishida-like complex shows up during this computation. We give a Hodge theoretic interpretation of this computation (for $\QQ$-coefficients). Our first aim is to describe the mixed Hodge structures of the groups $H^{k}(X, \QQ)$. Moreover, given a line bundle $L$, we want to describe the morphism
	$$ c \colon H^{k}(X,\QQ) \to H^{k+2}(X, \QQ)$$
	given by the cup product with the Chern class of $L$ in terms of the data of the fan $\cP$. 

    Note that $ \gr_{F}^{p} H^{k}(X, \QQ_{X}) \simeq \HH^{k-p}(X, \duBois_{X}^{p}).$
    Passing to the Grothendieck dual, we see that
    $$ \HH^{k-p}(X, \duBois_{X}^{p})\dual \simeq \HH^{n-k+p}(X, \bfR\SheafHom_{\cO_{X}}(\duBois_{X}^{p}, \omega_{X})), $$
    using the fact that $X$ is Cohen--Macaulay. Note that
    $$ \bfR\SheafHom_{\cO_{X}}(\duBois_{X}^{p}, \omega_{X}) \simeq \Ish_{X}^{n-p}$$
    and each term of the Ishida complex is $\bfR\Gamma$-acyclic by \cite{CoxLittleSchenck-ToricVar}*{Theorem 9.2.5}. This shows that
    $$ \HH^{n-k+p}(X, \Ish_{X}^{n-p}) = H^{n-k+p}(\Ish_{\cP}^{n-p}),$$
    since the right hand side is the complex obtained by taking the global sections in the complex $\Ish_{X}^{n-p}$. We point out that $\Ish_{\cP}^{n-p}$ is simply a complex of finite-dimensional vector spaces.

    For a line bundle (more generally, for a $\QQ$-Cartier divisor) $L$, the first Chern class of $L$ induces a morphism of (complexes of) mixed Hodge modules
    $$ c : \QQ_{X}^{H} \to \QQ_{X}^{H}(1)[2].$$
    We refer to \cite{RSW-Lyubeznik-Thom-Gysin}*{\S1.3} for this map. By taking cohomologies, we have a morphism $H^{k}(X, \QQ) \to H^{k+2}(X, \QQ)(1)$ between mixed Hodge structures. In particular, it induces
    $$ c : \gr_{F}^{p}H^{k}(X, \QQ) \to \gr_{F}^{p+1}H^{k+2}(X, \QQ). $$
    By taking the dual, we have the morphism
    $$ c\dual : \HH^{n-k+p-1}(X,\Ish_{X}^{n-p-1}) \to \HH^{n-k+p}(X,\Ish_{X}^{n-p}).$$
    The goal is to describe $c\dual$ purely in terms of the data of the fan $\cP$ when the line bundle $L$ is torus equivariant (more generally, when $L$ is a toric $\QQ$-Cartier divisor). We point out that this is not a serious assumption since every divisor on a toric variety is linearly equivalent to a torus equivariant one.
	
	\subsection{Total space of the line bundle} \label{subsec:total-space-line-bundle}
	Let $X$ be an $n$-dimensional toric variety and let $D = \sum \alpha_{\rho} S_{\rho}$ be an integral Cartier divisor on $X$. This means that for each maximal dimensional face $\sigma$, one has $u_{\sigma} \in M$ such that
	$$ \langle u_{\sigma}, \rho \rangle = \alpha_{\rho} \in \ZZ,$$
	where, for notational convenience, we identify $\rho$ with its primitive element in the ray. The total space $L\to X$ corresponding to $D$ is again a toric variety, and we describe this line bundle in terms of toric geometry, i.e., cones and fans.
	
	Let $\widetilde{N} = N \oplus \ZZ e_{n+1}$ and $\widetilde{M} = M \oplus \ZZ e_{n+1}\sta$. Let $V = M_{\RR}$ and $\widetilde{V} = \widetilde{M}_{\RR}$. For each maximal dimensional face $\sigma$, we consider
	$$ \widetilde{\sigma} = \{(x, t) \in N \oplus \ZZ e_{n+1} : t \geq u_{\sigma}(x) , x \in \sigma\}.$$
	We have a fan $\cQ$ in $\widetilde{N}$ whose maximal dimensional faces are $\widetilde{\sigma}$. We point out that the faces of $\cQ$ is either $\widetilde{\tau}$ for $\tau \in \cP$ (defined analogously as $\widetilde{\sigma}$), or
    $$ \widehat{\tau} = \{ (x, t) \in N \oplus \ZZ e_{n+1} : t = u_{\tau}(x), x \in \tau\}.$$
    We denote by $\widetilde{\cP} = \{ \widetilde{\tau} : \tau \in \cP\}$ and $\widehat{\cP} = \{ \widehat{\tau} : \tau \in \cP\}$. We get
    $$ \cQ_{l} = \widetilde{\cP}_{l-1} \cup \widehat{\cP}_{l}.$$
    Let the corresponding toric variety of the fan $\cQ$ be $L$. We clearly have a projection map $\pi \colon L \to X$ coming from the projection $\widetilde{N} \to N$.
	
	\begin{lemm} \label{lemm:description-total-space}
		$\pi \colon L \to X$ is the total space corresponding to the torus-invariant divisor $D$.
	\end{lemm}
	\begin{proof}
		First, we observe that for $\tau\in \ZZ$ and for $\phi \in M$,
		$$ \phi + \tau e_{n+1}\sta \in \widetilde{\sigma}\dual \cap \widetilde{M} \quad \text{if and only if} \quad \tau u_{\sigma} + \phi \in \sigma\dual \cap M, \text{ and } \tau \geq 0.$$
		Suppose $\phi + \tau e_{n+1}\sta \in \widetilde{\sigma}\dual$. First, $e_{n+1} \in \widetilde{\sigma}$, so $\tau \geq 0$. Also, $(x, u_{\sigma}(x)) \in \widetilde{\sigma}$ for $x \in \sigma$. This says $\phi(x) + u_{\sigma}(x)  \tau \geq 0$ for all $x \in \sigma$. Hence $\phi + \tau u_{\sigma} \in \sigma\dual$. The other direction can be verified similarly.
		Hence, we have
		\[ \widetilde{\sigma}\dual \cap \widetilde{M} \simeq (\sigma\dual \cap M) + \ZZ_{\geq 0} \cdot (-u_{\sigma} + e_{n+1}\sta). \]
		Let $U_{\widetilde{\sigma}}$ and $U_{\sigma}$ be the affine charts of $L$ and $X$ corresponding to the faces $\widetilde{\sigma}$ and $\sigma$, respectively. One can easily see that $\pi^{-1}(U_{\sigma}) = U_{\widetilde{\sigma}}$ and that the local description of the morphism $\pi : U_{\widetilde{\sigma}} \to U_{\sigma}$ is given by
		$$  \CC[\sigma\dual \cap M] \to \CC[\widetilde{\sigma}\dual \cap \widetilde{M}], \qquad \chi^{\phi} \mapsto \chi^{\phi}.$$
		We see that $U_{\widetilde{\sigma}}$ is isomorphic to $U_{\sigma} \times \CC$ from the description of $\widetilde{\sigma}\dual \cap \widetilde{M}$. We fix a non-vanishing section $s_{\sigma} : U_{\sigma} \to U_{\widetilde{\sigma}}$ given by
		$$ \CC[\widetilde{\sigma}\dual \cap \widetilde{M}] \to \CC[\sigma\dual \cap M], \qquad \chi^{\phi + t e_{n+1}\sta} \mapsto \chi^{\phi + t u_{\sigma}}.$$
		The multiplication on each fibers (i.e., the map $L \times \CC\to L$) is locally given by the morphisms
		$$ \CC[\widetilde{\sigma}\dual \cap \widetilde{M}] \to \CC[\ZZ_{\geq0}] \otimes \CC[\widetilde{\sigma}\dual \cap \widetilde{M}], \qquad \chi^{\phi} \mapsto \chi^{\langle e_{n+1}, \phi \rangle} \otimes \chi^{\phi}.$$
		We consider the overlaps. Let $\sigma_{1}$ and $\sigma_{2}$ be two maximal dimensional faces and let $\tau= \sigma_{1} \cap \sigma_{2}$. We compare the two non-vanishing sections $s_{\sigma_{1}}$ and $s_{\sigma_{2}}$, restricted to $U_{\tau}$. Note that the morphism given by
        \begin{align*}
           \CC [\widetilde{\tau}\dual \cap \widetilde{M}] \to \CC[\widetilde{\tau}\dual \cap \widetilde{M}] \otimes \CC[\ZZ_{\geq 0}] \xrightarrow{(s_{\sigma_1}, g)} \CC[\tau\dual \cap M], \\ \chi^{\phi + t e_{n+1}\sta} \mapsto \chi^{\phi + te_{n+1}\sta} \otimes \chi^{t}  \mapsto \chi^{\phi + tu_{\sigma_{1}}} \cdot \chi^{t(u_{\sigma_{2}} - \sigma_{1})} 
        \end{align*}
		is exactly $s_{\sigma_{2}}$. This shows that $s_{\sigma_{2}} = g \cdot s_{\sigma_{1}}$ where $g : U_{\tau} \to \CC\sta$ is given by the invertible function $\chi^{u_{\sigma_{2}} - u_{\sigma_{1}}} \in \CC[\tau\dual \cap M]$. This exactly show that $L$ is the line bundle corresponding to the divisor $D= \sum a_{\rho} S_{\rho}$.
	\end{proof}

	\subsection{The Atiyah class} \label{subsec:Atiyah-class-smooth}
	Let $X$ be a complex variety and $\pi \colon L \to X$ be a line bundle. We view the Chern class $c = c_{1}(L)$ as a morphism (see \cite{RSW-Lyubeznik-Thom-Gysin}*{\S1.3})
	$$ c \colon \QQ_{X}^{H} \to \QQ_{X}^{H}(1)[2],$$
	which gives $c\colon\duBois_{X}^{p} \to \duBois_{X}^{p+1}[1]$ after taking the graded de Rham complex (see \cite{LCDTV1}*{\S2.4}). This in particular gives an extension of $\duBois_{X}^{p}$ by $\duBois_{X}^{p+1}$. First, we show that if $X$ is smooth, the sequence
	$$ 0 \to \Omega_{X}^{p+1} \xrightarrow{i} \Omega_{L}^{p+1}(\log X)|_{X} \xrightarrow{r}\Omega_{X}^{p} \to 0 $$
	coincides with the extension class given by $c$. This can be checked using the description of the Atiyah class in terms of \v{C}ech cocycles $\{ d \log g_{\alpha\beta}\}$ \cite{Huybrechts-complexgeometry}*{Definition 4.2.18}. Here $\{g_{\alpha\beta}\}$ is the \v{C}ech cocycle defining the class of $L$. The map $i$ is given by $i(\alpha) = \pi\sta \alpha|_{X}$ and $r$ is given by the Poincaré residue map. On the open subset $U_{\alpha} \subset X$ where  $L$ is trivialized by the section $e_{\alpha}$, we have the local splitting $\phi_{\alpha}$ of $r$ on $U_{\alpha}$ given by
	$$  \beta \mapsto d\log e_{\alpha}\sta \wedge \beta $$
	where $e_{\alpha}\sta : L|_{U_{\alpha}} \to \CC$ is the dual of $e_{\alpha}$. Then one can easily see that $\phi_{\alpha} - \phi_{\beta}: \Omega_{U_{\alpha\beta}}^{p} \to \Omega_{U_{\alpha\beta}}^{p+1} $ is given by wedging with $d\log g_{\alpha\beta}$.
	
	It is an easy exercise that by applying $\SheafHom_{\cO_{X}}(\cdot, \omega_{X})$ to the sequence above, we get
	$$ 0 \to \Omega_{X}^{n-p} \to \Omega_{L}^{n-p} (\log X)|_{X} \to \Omega_{X}^{n-p-1} \to 0,$$
	with the same morphisms.
	
	We return to the case when $X$ is a (possibly singular) toric variety and $L$ is a line bundle on $X$. Let $X^{\circ}$ be the complement of the codimension 2 torus-invariant closed subvarieties of $X$ and denote by $j : X^{\circ} \to X$ the inclusion. One can see that $X^\circ$ is the (smooth) toric variety associated to the one-dimensional skeleton on the fan $\cP$ of $X$. The morphism $c : \duBois_{X}^{p} \to \duBois_{X}^{p+1}[1]$ gives us an extension
	$$ 0 \to  \duBois_{X}^{p+1} \to E \to \duBois_{X}^{p} \to 0.$$
	Note that $\duBois_{X}^{p}$ and $\duBois_{X}^{p+1}$ are both $\cS_{2}$, hence $E$ is also $\cS_{2}$ using the description of depth by $\SheafExt$-vanishing. Therefore, we have $E = j\lsta (E|_{X^{\circ}})$. Restricting the sequence above to $X^{\circ}$, we see that the short exact sequence becomes
	$$ 0 \to \Omega_{X^{\circ}}^{p+1} \to \Omega_{L^{\circ}}^{p+1}(\log X^{\circ})|_{X^{\circ}} \to \Omega_{X^{\circ}}^{p} \to 0,$$
	using the description of the Atiyah class on smooth varieties. From the discussion in \S\ref{section:diff-forms-on-toric}, this shows that we have a commutative diagram
    $$
    \begin{tikzcd}
        0 \ar[r] & \bigwedge^{p+1}V \otimes \cO_{X} \ar[r] \ar[d] & \bigwedge^{p+1} \widetilde{V} \otimes \cO_{X} \ar[r] \ar[d]& \bigwedge^{p} V \otimes \cO_{X} \ar[r] \ar[d]& 0 \\
        0 \ar[r] & \bigoplus_{\rho \in \cP_{1}} \bigwedge^{p}\widetilde{\rho}^{\perp} \otimes \cO_{D_{\rho}} \ar[r] & \bigoplus_{\rho \in \cP_{1}} \bigwedge^{p} \widehat{\rho}^{\perp} \otimes \cO_{D_{\rho}} \ar[r] & \bigoplus_{\rho \in \cP_{1}} \bigwedge^{p-1} \widetilde{\rho}^{\perp} \otimes \cO_{D_{\rho}} \ar[r] & 0 
    \end{tikzcd}
    $$
    and the kernels of the vertical arrows give the extension class
    $$ 0 \to \duBois_{X}^{p+1} \to E \to \duBois_{X}^{p} \to 0.$$

    \subsection{The Grothendieck dual of the Atiyah class} \label{subsec:Gro-dual-Atiyah-class}
    We finally describe the Grothendieck dual of the Atiyah class in terms of Ishida-type complexes. Before that, we set-up some notation. For each $\tau \in \cP$, we define integers
    $$ a_{\tau} := \# \frac{\langle \widetilde{\tau} \rangle \cap \widetilde{N}}{\ZZ \cdot e_{n+1} + \left( \langle \widehat{\tau} \rangle \cap \widetilde{N} \right)}$$
    as in Definition \ref{defi:a_mu}.
    \begin{rema} \label{rema:Gro-dual-Atiyah-linalg}
    From the discussion in Lemma \ref{lemm:vector_multiply_by_amu} and \ref{lemm:vector_multiply_amu2}, we have the followings:
    	\begin{enumerate}
    		\item $a_{\tau} n_{\widehat{\tau}, \widetilde{\tau}} = e_{n+1} \mod \langle \widehat{\tau}\rangle \cap \widetilde{N}$.
    		\item For $\mu \subset \tau$ in $\cP$ such that $d_{\mu}+1 = d_{\tau}$, we have $n_{\mu, \tau} = n_{\widetilde{\mu}, \widetilde{\tau}} \mod \langle \widetilde{\mu}\rangle \cap \widetilde{N}$.
    		\item $ a_{\mu} n_{\widetilde{\mu}, \widetilde{\tau}} = 
            a_{\tau} n_{\widehat{\mu}, \widehat{\tau}} \mod \langle \widetilde{\mu} \rangle \cap \widetilde{N} $.
    	\end{enumerate}
    \end{rema}
    
    From here, we see that we have a short exact sequence of complexes as follows:
    $$ \begin{tikzcd}
    	0 \ar[r] & \bigwedge^{p+1} V \otimes \cO_{X} \ar[r] \ar[d, "\varphi_{\widetilde{0},\widetilde{\rho}}"]  & \bigwedge^{p+1}\widetilde{V} \otimes \cO_{X} \ar[r, "\varphi_{\widehat{0}, \widetilde{0}}"] \ar[d, "\varphi_{\widehat{0}, \widehat{\rho}}"]& \bigwedge^{p}V \otimes \cO_{X} \ar[r] \ar[d, "\varphi_{\widetilde{0},\widetilde{\rho}}"] & 0 \\
    	0 \ar[r] & \bigoplus_{\rho \in \cP_{1}} \bigwedge^{p} \widetilde{\rho}^{\perp} \otimes \cO_{S_{\rho}} \ar[r, "a_{\rho}"] \ar[d] & \bigoplus_{\rho \in \cP_{1}} \bigwedge^{p} \widehat{\rho}^{\perp} \otimes \cO_{S_{\rho}} \ar[r, "-\varphi_{\widehat{\rho}, \widetilde{\rho}}"] \ar[d] & \bigoplus_{\rho \in \cP_{1}} \bigwedge^{p-1} \widetilde{\rho}^{\perp} \otimes \cO_{S_{\rho}} \ar[r] \ar[d] & 0 \\
    	& \vdots \ar[d] & \vdots \ar[d] & \vdots \ar[d] & \\
    	0 \ar[r] & \bigoplus_{\tau \in \cP_{l}} \bigwedge^{p+1-l} \widetilde{\tau}^{\perp} \otimes \cO_{S_{\tau}} \ar[r, "a_{\tau}"] \ar[d] &  \bigoplus_{\tau \in \cP_{l}} \bigwedge^{p+1-l} \widehat{\tau}^{\perp} \otimes \cO_{S_{\tau}} \ar[r, "(-1)^{l} \varphi_{\widehat{\tau}, \widetilde{\tau}}"] \ar[d]& \bigoplus_{\tau \in \cP_{l}} \bigwedge^{p-l} \widetilde{\tau}^{\perp} \otimes \cO_{S_{\tau}} \ar[r] \ar[d]& 0 \\
    	& \vdots & \vdots& \vdots  
    \end{tikzcd}$$
    Note that the first and third columns of this exact sequence are isomorphic to $\Ish_{X}^{p+1}$ and $\Ish_{X}^{p}$ respectively, by Remark \ref{rema:Gro-dual-Atiyah-linalg} (2). Note that $\widetilde{\tau}^{\perp} = \tau^{\perp} \subset V$. We denote the middle column by $\Ish_{X, D}^{p+1}$.

    \begin{prop} \label{prop:Gro-dual-Lefschetz-integral}
        As in the set-up of \S\ref{subsec:total-space-line-bundle}, the short exact sequence of complexes
        $$ 0 \to \Ish_{X}^{p+1} \to \Ish_{X, D}^{p+1} \to \Ish_{X}^{p} \to 0$$
        is Grothendieck dual to
        $$ 0 \to \duBois_{X}^{n-p} \to E \to \duBois_{X}^{n-p-1} \to 0,$$
        where the extension class of $\duBois_{X}^{n-p-1}$ by $\duBois_{X}^{n-p}$ is given by the Chern class $c_{1}(L)$ of $L$.
    \end{prop}
    \begin{proof}
        From Proposition \ref{prop:Gro-dual-of-Ishida-is-duBois}, we have $\bfR\SheafHom_{\cO_{X}} (\Ish_{X}^{p}, \omega_{X}) \simeq \duBois_{X}^{n-p}$. By dualizing the first exact sequence, we see that $\bfR \SheafHom_{\cO_{X}} (\Ish_{X, D}^{p+1}, \omega_{X})$ is also a sheaf, moreover, $\cS_{2}$. Hence, it remains to show the statement in a complement of a codimension 2 subset. Then we see that the assertion follows from the discussion in \S\ref{subsec:Atiyah-class-smooth}.
    \end{proof}

    \subsection{From $\ZZ$-divisors to $\QQ$-divisors}
    In this section, we assert that the same assertion works for $\QQ$-divisors as well. We consider $D= \sum \alpha_{\rho} S_{\rho}$ a $\QQ$-Cartier $\QQ$-divisor on $X$. Then for each face $\tau \in \cP$, we get $u_{\tau} \in M_{\QQ}$ such that $\langle u_{\tau}, \rho \rangle = \alpha_{\rho}$ for all rays $\rho \subset \tau$. Hence, we get a fan $\cQ$ in $\widetilde{N} = N \oplus \ZZ e_{n+1}$ in a similar manner as \S\ref{subsec:total-space-line-bundle}. We have a morphism $\pi : L \to X$ similarly, but this is \textit{not} a geometric line bundle on $X$. However, the Chern class $c_{1}(D)$ makes perfect sense, after multiplying by a certain integer to make $D$ integral and dividing back.
    
    Let $C$ be a positive integer such that $\sum C \alpha_{\rho} S_{\rho}$ is a Cartier $\ZZ$-divisor. For $\tau \in \cP$, we define
    \begin{align*}
        \widetilde{\tau}_C & :=  \{(x, t) \in N \oplus \ZZ e_{n+1} : t \geq C u_{\sigma}(x), x \in \tau \} \\
        \widehat{\tau}_C &:= \{ (x, t) \in N \oplus \ZZ e_{n+1} : t = C u_{\sigma}(x) , x \in \tau \}.
    \end{align*}
    We define
    $$ a_{\tau}:= \#\frac{\langle \widetilde{\tau} \rangle \cap \widetilde{N}}{\ZZ \cdot e_{n+1} + \left( \langle \widehat{\tau} \rangle \cap \widetilde{N} \right)}, \qquad b_{\tau} := \#\frac{\langle \widetilde{\tau}_C \rangle \cap \widetilde{N}}{\ZZ \cdot e_{n+1} + \left( \langle \widehat{\tau}_C \rangle \cap \widetilde{N} \right)}$$

    Here is our main lemma.
    \begin{lemm} \label{lemm:relate-C-times-and-Q-divisor}
        Let $\Ish_{X, D}^{p+1}$ and $\Ish_{X, CD}^{p+1}$ be defined analogously. We have an isomorphism of extensions
        $$ \begin{tikzcd}
            0 \ar[r] & \Ish_{X}^{p+1} \ar[r] \ar[d, equal] & \Ish_{X, D}^{p+1} \ar[r] \ar[d, "\varphi", "\simeq"'] & \Ish_{X}^{p} \ar[r]\ar[d, equal] & 0 \\
            0 \ar[r] & \Ish_{X}^{p+1} \ar[r] & \Ish_{X, CD} \ar[r, "C"] & \Ish_{X}^{p} \ar[r] & 0.
        \end{tikzcd} $$
        The $C$ in the bottom row means that the constant $C$ is multiplied from the one in Proposition \ref{prop:Gro-dual-Lefschetz-integral}.
    \end{lemm}

    This immediately shows the $\QQ$-divisor version of Proposition \ref{prop:Gro-dual-Lefschetz-integral}.

    \begin{coro} \label{coro:Gro-dual-Ishida-main-Q-div}
        Let $X$ be a proper toric variety and $D$ be a $\QQ$-Cartier $\QQ$-divisor on $X$. Then
        $$ 0 \to \Ish_{X}^{p+1} \to \Ish_{X, D}^{p+1} \to \Ish_{X}^{p} \to 0$$
        is Grothendieck dual to the extension class
        $$ 0 \to \duBois_{X}^{n-p} \to E \to \duBois_{X}^{n-p-1} \to 0$$
        given by the Chern class $c_{1}(D)$ of $D$.
    \end{coro}

    \begin{proof}[Proof of Lemma \ref{lemm:relate-C-times-and-Q-divisor}]
        We describe the morphism $\phi$ term by term. Before that, we define
        $$ i_{C} : \widetilde{V} \to \widetilde{V}, \qquad \phi + t e_{n+1}\sta \mapsto \phi + C^{-1}t e_{n+1}\sta, \quad \text{for } \phi \in V.$$
        We first point out that
        \begin{align*}
            \widehat{\tau}^{\perp} & = \{ \phi + t e_{n+1}\sta : t \langle u_{\tau}, x\rangle+ \phi(x) = 0, \quad \text{for all } x \in \tau \} \\
            \widehat{\tau}_{C}^{\perp} & = \{ \phi + t e_{n+1}\sta : t C\langle u_{\tau}, x\rangle+ \phi(x) = 0, \quad \text{for all } x \in \tau \}.
        \end{align*}
        This description easily shows that $i_{C}$ sends $\widehat{\tau}^{\perp}$ isomorphically to $\widehat{\tau}_{C}^{\perp}$. The morphism $\phi$ is defined as
        $$ \begin{tikzcd}
            \bigwedge^{p+1} \widetilde{V} \otimes \cO_{X} \ar[r] \ar[d, "i_{C}"] & \bigoplus_{\rho \in \cP_{1}} \bigwedge^{p} \widehat{\rho}^{\perp} \otimes \cO_{S_{\rho}} \ar[r] \ar[d, "a_{\rho}^{-1}b_{\rho} i_{C}"] & \cdots \ar[r] \ar[d] & \bigoplus_{\tau \in \cP_{l}} \bigwedge^{p+1-l} \widehat{\tau}^{\perp} \otimes \cO_{S_{\tau}} \ar[d, "a_{\tau}^{-1} b_{\tau} i_{C}"] & \cdots \\
            \bigwedge^{p+1} \widetilde{V} \otimes \cO_{X} \ar[r] & \bigoplus_{\rho\in \cP_{1}} \bigwedge^{p} \widehat{\rho}_{C}^{\perp} \otimes \cO_{S_{\rho}} \ar[r] & \cdots \ar[r] & \bigoplus_{\tau \in \cP_{l}} \bigwedge^{p+1-l} \widehat{\tau}_{C}^{\perp} \otimes \cO_{S_{\tau}} \ar[r] & \cdots.
        \end{tikzcd} $$
        Here, the vertical arrows are defined term by term. For each face $\tau$, we have the morphisms $a_{\tau}^{-1} b_{\tau} i_{C} : \widehat{\tau}^{\perp} \to \widehat{\tau}_{C}^{\perp}$ tensored by identities on $\cO_{S_{\tau}}$. First, we show that this is indeed a homomorphism of chain complexes. For this, it is enough to fix $\mu \subset \tau$ in $\cP$ such that $d_{\tau} = d_{\mu} + 1$ and consider the commutativity of the diagram
        $$ \begin{tikzcd}
            \bigwedge^{p+1 - d_{\mu}} \widehat{\mu}^{\perp} \ar[r, "\varphi_{\widehat{\mu}, \widehat{\tau}}"] \ar[d, "a_{\mu}^{-1} b_{\mu} i_{C}"] & \bigwedge^{p+1 - d_{\tau}} \widehat{\tau}^{\perp} \ar[d, "a_{\tau}^{-1} b_{\tau} i_{C}"] \\
            \bigwedge^{p+1 - d_{\mu}} \widehat{\mu}_{C}^{\perp} \ar[r, "\varphi_{\widehat{\mu}_{C}, \widehat{\tau}_{C}}"]& \bigwedge^{p+1 - d_{\tau}} \widehat{\tau}_{C}^{\perp} 
        \end{tikzcd}$$
        Note that the commutativity follows from Remark \ref{rema:Gro-dual-Atiyah-linalg}, since we have $n_{\mu, \tau} = a_{\mu} a_{\tau}^{-1} n_{\widehat{\mu}, \widehat{\tau}} = b_{\mu} b_{\tau}^{-1} n_{\widehat{\mu}_{C}, \widehat{\tau}_{C}} \mod \langle \widetilde{\mu} \rangle \cap \widetilde{N}$.

        The commutativity of the left square almost follows by definition. Indeed, the diagram
        $$ \begin{tikzcd}
            \bigwedge^{p+1 - d_{\mu}} \widetilde{\mu}^{\perp} \ar[r, "a_{\mu}"] \ar[d, equal] & \bigwedge^{p+1-d_{\mu}} \widehat{\mu}^{\perp} \ar[d, "a_{\mu}^{-1} b_{\mu} i_{C}"] \\ \bigwedge^{p+1-d_{\mu}} \widetilde{\mu}^{\perp} \ar[r, "b_{\mu}"] & \bigwedge^{p+1 - d_{\mu}} \widehat{\mu}^{\perp}
        \end{tikzcd}$$
        commutes. For the right square, it is enough to check that
        $$ \begin{tikzcd}
            \bigwedge^{p+1-d_{\mu}} \widehat{\mu}^{\perp} \ar[r, "\varphi_{\widehat{\mu}, \widetilde{\mu}}"] \ar[d, "a_{\mu}^{-1} b_{\mu} i_{C}"] & \bigwedge^{p-d_{\mu}} \widetilde{\mu}^{\perp} \ar[d, equal] \\ \bigwedge^{p+1-d_{\mu}} \widehat{\mu}_{C}^{\perp} \ar[r, "C \cdot \varphi_{\widehat{\mu}_{C}, \widetilde{\mu}_{C} }"] & \bigwedge^{p-d_{\mu}} \widetilde{\mu}^{\perp}
        \end{tikzcd}$$
        commutes. This follows from the fact that $a_{\mu} n_{\widehat{\mu}, \widetilde{\mu}} = e_{n+1} \mod \langle \widehat{\mu} \rangle \cap \widetilde{N}$ and $b_{\mu} n_{\widehat{\mu}_{C}, \widetilde{\mu}_{C}} = e_{n+1} \mod \langle \widehat{\mu_{C}} \rangle \cap \widetilde{N}$ (see Remark \ref{rema:Gro-dual-Atiyah-linalg}).
    \end{proof}

    As a corollary, we are able to describe the morphisms $c_{1}(D)\dual : \HH^{l}(X, \Ish_{X}^{p}) \to \HH^{l+1}(X,\Ish_{X}^{p+1})$ fully in terms of the data of the fan $\cP$.

    \begin{coro} \label{coro:connecting-map-is-Lefschetz}
        The morphism $c_{1}(D)\dual : \HH^{l}(X, \Ish_{X}^{p}) \to \HH^{l+1}(X, \Ish_{X}^{p+1})$ is induced by the connecting homomorphism of the cohomologies induced by the following short exact sequence of complexes.
        $$ \begin{tikzcd}
    	0 \ar[r] & \bigwedge^{p+1} V \ar[r] \ar[d, "\varphi_{\widetilde{0},\widetilde{\rho}}"]  & \bigwedge^{p+1}\widetilde{V} \ar[r, "\varphi_{\widehat{0}, \widetilde{0}}"] \ar[d, "\varphi_{\widehat{0}, \widehat{\rho}}"]& \bigwedge^{p}V \ar[r] \ar[d, "\varphi_{\widetilde{0},\widetilde{\rho}}"] & 0 \\
    	0 \ar[r] & \bigoplus_{\rho \in \cP_{1}} \bigwedge^{p} \widetilde{\rho}^{\perp} \ar[r, "a_{\rho}"] \ar[d] & \bigoplus_{\rho \in \cP_{1}} \bigwedge^{p} \widehat{\rho}^{\perp} \ar[r, "-\varphi_{\widehat{\rho}, \widetilde{\rho}}"] \ar[d] & \bigoplus_{\rho \in \cP_{1}} \bigwedge^{p-1} \widetilde{\rho}^{\perp} \ar[r] \ar[d] & 0 \\
    	& \vdots \ar[d] & \vdots \ar[d] & \vdots \ar[d] & \\
    	0 \ar[r] & \bigoplus_{\tau \in \cP_{l}} \bigwedge^{p+1-l} \widetilde{\tau}^{\perp}  \ar[r, "a_{\tau}"] \ar[d] &  \bigoplus_{\tau \in \cP_{l}} \bigwedge^{p+1-l} \widehat{\tau}^{\perp}  \ar[r, "(-1)^{l} \varphi_{\widehat{\tau}, \widetilde{\tau}}"] \ar[d]& \bigoplus_{\tau \in \cP_{l}} \bigwedge^{p-l} \widetilde{\tau}^{\perp} \ar[r] \ar[d]& 0 \\
    	& \vdots & \vdots& \vdots  
    \end{tikzcd}$$
    \end{coro}
    \begin{proof}
        This follows from the fact that structure sheaves of proper toric varieties are $\bfR \Gamma$-acyclic. Then the assertion can be immediately obtained from Corollary \ref{coro:Gro-dual-Ishida-main-Q-div}.
    \end{proof}

    \subsection{Towards local cohomological dimension}
    We now focus on the case when $X$ is a projective toric variety, and $D = \sum_{\rho} \alpha_{\rho} D_{\rho}$ is an ample $\QQ$-divisor on $X$. This means that the function
    $$ \psi : N_{\RR} \to \RR ,\qquad x \mapsto \langle u_{\tau}, x \rangle \text{ if } x \in \tau$$
    is strictly convex. Therefore,
    $$ \varsigma = \{ (x, t) \in N \oplus \RR e_{n+1} : t \geq \psi(x) \} $$
    is a strictly convex rational polyhedral cone in $\widetilde{N}_{\RR} = N_{\RR} \oplus \RR e_{n+1}$. Note that the non-trivial faces of $\varsigma$ are $\varsigma$ itself, and $\widehat{\tau}$ for $\tau \in \cP$. As a quick application of Corollary \ref{coro:connecting-map-is-Lefschetz}, we have the following:

    \begin{prop} \label{prop:exactness-of-Ish-and-lefschetz}
        Suppose that $p < n$. The following are equivalent:
        \begin{enumerate}
            \item $H^{l}(\Ish_{\varsigma}^{p+1}) = 0$
            \item $c_{1}(D)\dual \colon \HH^{l}(X, \Ish_{X}^{p}) \to \HH^{l+1}(X, \Ish_{X}^{p+1})$ is injective and $c_{1}(D)\dual \colon \HH^{l-1}(X, \Ish_{X}^{p}) \to \HH^{l}(X, \Ish_{X}^{p+1})$ is surjective.
            \item $c_{1}(D) \colon H^{n-l-1}(X, \duBois_{X}^{n-p-1}) \to H^{n-l}(X, \duBois_{X}^{n-p})$ is surjective and $c_{1}(D) \colon H^{n-l}(X, \duBois_{X}^{n-p-1}) \to H^{n-l+1}(X, \duBois_{X}^{n-p})$ is injective.
        \end{enumerate}
    \end{prop}
    \begin{proof}
        This follows from the fact that $\Ish_{\varsigma}^{p+1}$ is the middle term in the short exact sequence of Corollary \ref{coro:connecting-map-is-Lefschetz}.
    \end{proof}

    Until now in this section, we started from a projective toric variety $X$ with an ample $\QQ$-line bundle, and constructed an affine toric variety corresponding to the cone $\varsigma$ of one dimension higher. However, we can reverse the order of this, i.e., we can start from an affine toric variety corresponding to a full-dimensional cone $\varsigma$ and consider a ray $\rho$ in the interior. By performing a $\ZZ$-linear change of coordinates, we can assume that $\rho$ is one of the basis vectors of $N$, and we get a projective toric variety of dimension one less, with an ample line bundle. In this way, we can control the vanishing and non-vanishing behavior of the cohomologies of $\Ish_{\varsigma}$ in terms of the Lefschetz operator on the singular cohomology of a projective toric variety of dimension one less. In particular, rephrasing Corollaries \ref{coro:Gro-dual-Ishida-main-Q-div} and \ref{coro:connecting-map-is-Lefschetz} immediately gives a proof of Theorem \ref{theo:lefschetz-ishida-correct-ver}. 
    
    We now give a proof of Proposition \ref{prop:hard-lefschetz-injectivity}. We would like to thank Kalle Karu for suggesting comparing certain graded pieces of the singular cohomology with intersection cohomology.

    \begin{proof}
        We will show that for every $p$, the natural map $\gr^p_F H^{2p}(X,\QQ) \to \gr^p_F IH^{2p}(X,\QQ)$ is injective. Note that we have $\gr^p_F H^{2p}(X) \simeq H^{p}(X, \duBois_{X}^{p})$. The above injectivity would suffice since by the Hard Lefschetz theorem for intersection cohomology, $\gr^p_F IH^{2p}(X,\QQ) \to \gr^p_F IH^{2p+2}(X,\QQ)$ is injective for $p \leq \frac{n-1}{2}$.

        First, by Weber's theorem \cite[Remark 6.5]{Popa-Park:lefschetz} (see also \cite[Theorem 1.8]{Weber}), we have
        \[ \ker\left(H^{2p}(X,\QQ) \to IH^{2p}(X,\QQ) \right) = W_{2p-1} H^{2p}(X,\QQ). \]
        By \cite[Corollary 1.2]{LCDTV1}, $H^{2p}(X,\QQ)$, and hence $W_{2p-1} H^{2p}(X,\QQ)$, is mixed of Hodge--Tate type, i.e. all the weight graded pieces are pure Hodge structures of Hodge--Tate type. Therefore we have $\gr^p_F W_{2p-1} H^{2p}(X,\QQ) = 0$. We can see this by taking the short exact sequences $0 \to W_{i-1} H^{2p}(X,\QQ) \to W_i H^{2p}(X,\QQ) \to \gr^W_i H^{2p}(X,\QQ) \to 0$ for $i<2p$ and then using the fact that $\gr^p_F \gr^W_i H^{2p}(X,\QQ) = 0$ since $\gr^W_i H^{2p}(X,\QQ)$ is of Hodge--Tate type of weight $i$. Therefore, the natural map
        $$ \gr^p_F H^{2p}(X,\QQ) \to \gr^p_F IH^{2p}(X,\QQ)$$
        is injective, which finishes the proof.
    \end{proof}

    We end this section by giving the proof of Corollary \ref{coro:lcdef-of-4dim}.
    \begin{proof}[Proof of Corollary \ref{coro:lcdef-of-4dim}]
        We see that $\lcdef(X)$ can only be zero or 1, and it is 1 if and only if $H^{2}(\Ish_{\sigma}^{3}) \neq 0$. Note that $\HH^{2}(\Ish_{E}^{3}) \simeq H^{1}(\cO_{E})\dual = 0$ and hence $H^{2}(\Ish_{\sigma}^{3})$ is the kernel of the surjective map
        $$ \HH^{2}(\Ish_{E}^{2}) \to \HH^{3}(\Ish_{E}^{3})$$
        since the next term $H^3(\Ish^3_\sigma) = 0$ by \cite{LCDTV1}*{Theorem 1.3}. Note that $\HH^{3}(\Ish_{E}^{3}) \simeq H^{0}(E, \cO_{E})\dual$ is one dimensional. Also, $\HH^{3}(\Ish_{E}^{1}) = H^{0}(E, \duBois_{E}^{2})\dual = 0$ since $\Ish_{E}^{1}$ is supported in degrees 0 and 1. Also, $H^{2}(E, \cO_{E}) = 0$ since $E$ is toric. Therefore, $H^{2}(E,\CC) \simeq H^{1}(E, \duBois_{E}^{1}) \simeq \HH^{2}(E, \Ish_{E}^{2})\dual$. This shows the assertion.
    \end{proof}

    \section{Other results on the local cohomological defect}\label{sec:combinatorial-results-lcdef}

    In this section, we prove the combinatorial results stated at the end of the introduction.
	
	\begin{proof}[Proof of Theorem \ref{theo:adding-new-ray}]
	We observe that the affine chart of $\widetilde{X}$ corresponding to $\sigma$ is isomorphic to $\CC^{\times} \times X$, and hence $\lcdef(\widetilde{X}) \geq \lcdef(X)$. Therefore, it is enough to show the other inequality. We put $c = \lcdef(X)$.
    
    Let $\cP$ be the fan consisting of faces of $\sigma$ and $\cQ$ the fan corresponding to $\prescript{\rho}{}{\sigma}$. We first notice that the elements in $\cQ$ are either $\mu \in \cP$, or $\prescript{\rho}{}{\mu}:= \mathrm{span}_{\RR\geq0}(\rho, \mu)$ for some $\mu \in \cP$. We let $V = M_{\RR}$ and $\widetilde{V} = \widetilde{M}_{\RR}$. For $\mu \in \cP$, we denote by
	$$ \mu_{V}^{\perp} = \{ u \in V : \langle u ,v \rangle = 0 \text{ for all } v \in \mu\} $$
	and $\mu_{\widetilde{V}}^{\perp} = \{ u \in \widetilde{V} : \langle u ,v \rangle = 0 \text{ for all } v \in \mu\}$ in order to prevent confusion. Note that we have the short exact sequence
	$$ 0 \to \sigma_{\widetilde{V}}^{\perp} \to \widetilde{V} \to V \to 0$$
	where the right map is the restriction. We also have $\widetilde{V} = \rho^{\perp} \oplus \sigma_{\widetilde{V}}^{\perp}$ and emphasize that this is an \textit{internal} direct sum. Similarly, we have $\mu_{\widetilde{V}}^{\perp} = \sigma_{\widetilde{V}}^{\perp} \oplus \prescript{\rho}{}{\mu}^{\perp}$ for each $\mu \in \cP$. We also point out that the restriction morphism $\mu_{\widetilde{V}}^{\perp} \to \mu_{V}^{\perp}$ sends $\prescript{\rho}{}{\mu}^{\perp}$ isomorphically to $\mu_{V}^{\perp}$. We recall the integer $a_{\mu}$'s in Definition \ref{defi:a_mu}.
	
	\begin{lemm} \label{lemm:adding-ray-linalg-1}
		Let $\mu, \nu \in \cP$ such that $\mu \subset \nu$ and $d_{\nu} = d_{\mu} + 1$. Then the following diagram commutes:
		$$ \begin{tikzcd}
			\bigwedge^{l} \mu_{\widetilde{V}}^{\perp} \ar[r, "a_{\mu}^{-1}"] \ar[d, "\varphi_{\mu, \nu}"] & \sigma_{\widetilde{V}}^{\perp} \otimes \bigwedge^{l-1} \prescript{\rho}{}{\mu}^{\perp} \oplus \bigwedge^{l} \prescript{\rho}{}{\mu}^{\perp} \ar[d, "\varphi_{\prescript{\rho}{}{\mu}, \prescript{\rho}{}{\nu}}"] \\
			\bigwedge^{l-1} \nu_{\widetilde{V}}^{\perp} \ar[r, "a_{\nu}^{-1}"] & \sigma_{\widetilde{V}}^{\perp} \otimes \bigwedge^{l-1} \prescript{\rho}{}{\nu}^{\perp} \oplus \bigwedge^{l} \prescript{\rho}{}{\nu}^{\perp},
		\end{tikzcd} $$
		where the horizontal arrows are induced by $\mu_{\widetilde{V}}^{\perp} = \sigma_{\widetilde{V}}^{\perp} \oplus \prescript{\rho}{}{\mu}^{\perp}$ followed by multiplication by $a_{\mu}^{-1}$, and similarly for $\nu$.
	\end{lemm}
	\begin{proof}
		This follows by the fact $a_{\mu} n_{\mu, \nu} = a_{\nu} n_{\prescript{\rho}{}{\mu}, \prescript{\rho}{}{\nu}}$ mod $\langle \prescript{\rho}{}{\mu}\rangle \cap \widetilde{N}$, addressed in Lemma \ref{lemm:vector_multiply_amu2}.
	\end{proof}
	
	\begin{lemm} \label{lemm:adding-ray-linalg-restriction}
		In the same setting as the previous lemma, the following diagram commutes:
		$$ \begin{tikzcd}
			\bigwedge^{l} \prescript{\rho}{}{\mu}^{\perp} \ar[r, "a_{\mu}"] \ar[d, "\varphi_{\prescript{\rho}{}{\mu}, \prescript{\rho}{}{\nu}}"] & \bigwedge^{l}\mu_{V}^{\perp} \ar[d, "\varphi_{\mu, \nu}"] \\ \bigwedge^{l-1} \prescript{\rho}{}{\nu}^{\perp} \ar[r, "a_{\nu}"] & \bigwedge^{l-1} \nu_{V}^{\perp},
		\end{tikzcd}$$
		where the horizontal arrows are given by the restriction morphisms, followed by multiplication by $a_{\mu}$ and $ a_{\nu}$, respectively.
	\end{lemm}
	\begin{proof}
		Analogous to the previous Lemma.
	\end{proof}
	
	We finally give the proof of Theorem \ref{theo:adding-new-ray}. Since $\lcdef(X) = c$, we have $\cH^{i}(\Ish_{X}^{l}) = 0$ for $i \geq n-l+c+1$, by Proposition \ref{prop:lcd-intermsof-Ishida}. It is enough to show that $\cH^{i}(\Ish_{\widetilde{X}}^{l}) = 0$ for $i \geq (n+1)-l+c+1$. Again, we use the grading of the Ishida complex by $\widetilde{M} = \Hom(\widetilde{N}, \ZZ)$ and examine the exactness at each degree.
	
	\noindent
	We first examine the case when $u \in \mu_{\circ}\sta \subset \widetilde{M}$ for some $\mu \in \cP$. We see that $\Ish_{\widetilde{X}}^{l}$ at degree $u$ is
	$$ Q^{\bullet} : \bigwedge^{l} \widetilde{V} \to \bigoplus_{\lambda\lind{1} \in \cP_{1}^{\subset \mu}} \bigwedge^{l-1} (\lambda\lind{1})_{\widetilde{V}}^{\perp} \to \ldots \bigoplus_{\lambda\lind{l}\in \cP_{l}^{\subset \mu}} \RR_{\lambda\lind{l}}.$$
	Observe that from Lemma \ref{lemm:adding-ray-linalg-1}, this complex decomposes into two pieces $Q^{\bullet} \simeq Q_{1}^{\bullet} \oplus Q_{2}^{\bullet}$, where the individual complexes are given by
	\begin{align*}
		Q_{1}^{\bullet} &: \sigma_{\widetilde{V}}^{\perp} \otimes \left( \bigwedge^{l-1}\rho^{\perp} \to \bigoplus_{\lambda\lind{1}\in \cP_{1}^{\subset \mu}} \bigwedge^{l-2} \prescript{\rho}{}{\lambda\lind{1}}^{\perp} \to \ldots \to \bigoplus_{\lambda\lind{l-1}\in \cP_{l-1}^{\subset \mu}}  \RR_{\prescript{\rho}{}{\lambda\lind{l-1}}} \right), \\
		Q_{2}^{\bullet} &: \bigwedge^{l}\rho^{\perp} \to \bigoplus_{\lambda\lind{1}\in \cP_{1}^{\subset \mu}} \bigwedge^{l-1} \prescript{\rho}{}{\lambda\lind{1}}^{\perp} \to \ldots \to \bigoplus_{\lambda\lind{l}\in \cP_{l}^{\subset \mu}}  \RR_{\prescript{\rho}{}{\lambda\lind{l}}}.\\
	\end{align*}
	We point out that the isomorphism $Q^{\bullet} \simeq Q_{1}^{\bullet} \oplus Q_{2}^{\bullet}$ is given by multiplication by $a_{\lambda}^{-1}$ for each term corresponding to $\lambda$. By Lemma \ref{lemm:adding-ray-linalg-restriction}, the complexes $Q_{1}^{\bullet}$ and $Q_{2}^{\bullet}$ are isomorphic to the degree $v$ part of the complex $\sigma_{\widetilde{V}}^{\perp} \otimes \Ish_{X}^{l-1}$ and $\Ish_{X}^{l}$, respectively, where $v \in \mu_{\circ}\sta \subset M$. The multiplications by $a_{\lambda}$'s are also involved here as well. Therefore, this complex is exact in cohomological degrees $i \geq n-(l-1)+c+1$.
	
	\noindent
	Now, we examine the case when $u \in \prescript{\rho}{}{\mu}_{\circ}\sta$ for some $\mu \in \cP$. Then $\Ish_{\widetilde{X}}^{l}$ at degree $u$ is
	$$ C^{\bullet} : \bigwedge^{l} \widetilde{V} \to \bigoplus_{\lambda\lind{1} \in \cQ_{1}^{\subset \prescript{\rho}{}{\mu}}} \bigwedge^{l-1} (\lambda\lind{1})_{\widetilde{V}}^{\perp} \to \ldots \bigoplus_{\lambda\lind{l}\in \cQ_{l}^{\subset \prescript{\rho}{}{\mu}}} \RR_{\lambda\lind{l}}.$$
	We have a surjective homomorphism between chain complexes $C^{\bullet} \to Q^{\bullet}$ given by the projection, where $Q^{\bullet}$ is describe above. The kernel of this morphism is given by
	$$ S^{\bullet} : 0 \to \bigwedge^{l-1} \rho^{\perp} \to \bigoplus_{\lambda\lind{1} \in \cP_{1}^{\subset \mu}} \bigwedge^{l-2} \prescript{\rho}{}{\lambda}\lind{1}^{\perp} \to \ldots \to \bigoplus_{\lambda\lind{l-1} \in \cP_{l-1}^{\subset \mu}} \RR_{\prescript{\rho}{}{\lambda}\lind{l-1}}, $$
	where $\bigwedge^{l-1} \rho^{\perp}$ is sitting in cohomological degree 1. We describe the connecting homomorphism $H^{i}(Q^{\bullet}) \to H^{i+1}(S^{\bullet})$. For this, we consider $Q^{\bullet} \simeq Q_{1}^{\bullet} \oplus Q_{2}^{\bullet}$ above, and compute the individual connecting homomorphisms $\delta_{1} : H^{i}(Q_{1}^{\bullet}) \to H^{i+1}(S^{\bullet})$ and $\delta_{2} : H^{i}(Q_{2}^{\bullet}) \to H^{i+1}(S^{\bullet})$.
	
	\noindent
	We first show that $\delta_{2}$ is the zero map. Pick an element $(v_{\lambda})_{\lambda \in \cP_{i}^{\subset\mu}}$ representing a cohomology class in $H^{i}(Q_{2}^{\bullet})$. Here, we have $v_{\lambda} \in \bigwedge^{l-i} \prescript{\rho}{}{\lambda}^{\perp}$. This element lifts to an element in $C^{i} = \bigoplus_{\lambda \in \cQ_{i}^{\subset \prescript{\rho}{}{\mu}}} \bigwedge^{l-i} \lambda^{\perp}$ by assigning zero to the faces containing $\rho$, and $a_{\lambda} v_{\lambda}$ for the faces $\lambda$ that does not contain $\rho$. The image $w$ of this element in $C^{i+1}$ should lie in $S^{i+1}$ and it represents $\delta_{2}((a_{\lambda}v_{\lambda})_{\lambda \in \cP_{i}^{\subset\mu}})$. Let $w_{\nu}$ be the $\bigwedge^{l-i-1}\prescript{\rho}{}{\nu}^{\perp}$-component of $w$ where $\nu \in \cP_{i}^{\subset\mu}$. Then we have
	$$ w_{\nu} = \varphi_{\nu, \prescript{\rho}{}{\nu}}(a_{\nu}v_{\nu}).$$
	However, the kernel of the map $\varphi_{\nu, \prescript{\rho}{}{\nu}} : \bigwedge^{l-i}\nu_{\widetilde{V}}^{\perp} \to \bigwedge^{l-i-1} \prescript{\rho}{}{\nu}^{\perp} $ is exactly $\bigwedge^{l-i} \prescript{\rho}{}{\nu}$ and therefore $w_{\nu} = 0$. This shows that the map $\delta_{2} : H^{i}(Q_{2}^{\bullet}) \to H^{i+1}(S^{\bullet})$ is zero.
	
	\noindent
	Next, we show that $\delta_{1}$ induces an isomorphism. Similarly we pick $(v_{\lambda})_{\lambda \in \cP_{i}^{\subset\mu}}$ representing a cohomology class in $H^{i}(Q_{1}^{\bullet})$, where $v_{\lambda} \in \sigma_{\widetilde{V}}^{\perp} \otimes \bigwedge^{l-i-1} \lambda^{\perp}$. By a similar computation, $\delta_{1}((v_{\lambda})_{\lambda \in \cP_{i}^{\subset\mu}})$ is represented by $(w_{\lambda})_{\lambda \in \cP_{i}^{\subset\mu}}$, where $w_{\lambda} = \varphi_{\lambda, \prescript{\rho}{}{\lambda}}(a_{\lambda}v_{\lambda})$. Composing Lemmas \ref{lemm:adding-ray-linalg-1} and \ref{lemm:anti-commutativity-complex}, we see that the following diagram commutes:
	$$ 
	\begin{tikzcd}
		\sigma_{\widetilde{V}}^{\perp} \otimes \bigwedge^{l-i-1}\prescript{\rho}{}{\lambda}^{\perp} \ar[r, hook, "a_{\lambda}"]\ar[d,"\id \otimes \varphi_{\prescript{\rho}{}{\lambda}, \prescript{\rho}{}{\nu}}"] & \bigwedge^{l-i} \lambda_{\widetilde{V}}^{\perp} \ar[r, "\varphi_{\lambda, \prescript{\rho}{}{\lambda}}"] \ar[d,"\varphi_{\lambda, \nu}"]&\bigwedge^{l-i-1}\prescript{\rho}{}{\lambda}^{\perp} \ar[d,"-\varphi_{\prescript{\rho}{}{\lambda}, \prescript{\rho}{}{\nu}}"] \\
		\sigma_{\widetilde{V}}^{\perp} \otimes \bigwedge^{l-i-1}\prescript{\rho}{}{\nu}^{\perp} \ar[r, hook, "a_{\nu}"] & \bigwedge^{l-i} \nu_{\widetilde{V}}^{\perp} \ar[r, "\varphi_{\nu, \prescript{\rho}{}{\nu}}"] &\bigwedge^{l-i-1}\prescript{\rho}{}{\nu}^{\perp},
	\end{tikzcd}
	$$
	for $\lambda \in \cP_{i}$ and $\nu \in \cP_{i+1}$ such that $\lambda \subset \nu$. This shows that $\delta_{1}$ is an isomorphism. Since we have the surjectivity of the connecting homomorphism $\delta : H^{i}(Q^{\bullet}) \to H^{i+1}(S^{\bullet})$, the long exact sequence of the cohomology associated to the short exact sequence $0 \to S^{\bullet} \to C^{\bullet} \to Q^{\bullet} \to 0$ splits into short exact sequences
	$$ 0\to H^{i}(C^{\bullet}) \to H^{i}(Q^{\bullet}) \to H^{i+1}(S^{\bullet}) \to 0,$$
	and we have $H^{i}(C^{\bullet}) \simeq H^{i}(Q_{2}^{\bullet})$. We already observed that $Q_{2}$ is the degree $v$ part of the complex $\Ish_{X}^{l}$, where $v \in \mu_{\circ}\sta \subset M$. Therefore, $H^{i}(C^{\bullet}) = 0$ for $i \geq n-l+c+1$. This concludes the proof of Theorem \ref{theo:adding-new-ray}.
    \end{proof}

    We now move on to the proofs of Proposition \ref{prop:Euler-char-lcdef}, and Theorems \ref{theo:shelling_contains_vertex} and \ref{theo:lcdef-of-not-a-pyramid-dim-4}. For $\sigma$ a full-dimensional cone of dimension 4, observe that $\lcdef(X)$ is either 0 or 1, and $\lcdef(X) = 1$ if and only if $H^{2}(\Ish_{\sigma}^{3}) \neq 0$.

    \begin{proof}[Proof of Proposition \ref{prop:Euler-char-lcdef}]
        Note that the only possible non-zero cohomologies of $\Ish_{\sigma}^{3}$ are $H^{1}$ and $H^{2}$. Let $v$ (respectively, $e$ and $f$) be the number of faces of $\sigma$ of dimension 1 (respectively, 2 and 3). Note that the Euler characteristic of $\Ish_{\sigma}^{3}$ is
        $$ 4 - 3v + 2e - f = 2(v - e+ f) - 3v + 2e - f = f - v.$$
        This quantity is equal to $\dim H^{2}(\Ish_{\sigma}^{3}) - \dim H^{1} (\Ish_{\sigma}^{3})$. Therefore, $H^{2}(\Ish_{\sigma}^{3}) \neq 0$ if $f > v$.
    \end{proof}

    Before proving Theorems \ref{theo:shelling_contains_vertex} and \ref{theo:lcdef-of-not-a-pyramid-dim-4}, we explain the following technique that will be commonly used. Let $\sigma$ be a full-dimensional cone of dimension 4. We prove the two theorems using the shelling of the cone $\sigma$.
    
    Let $f_{0},\ldots, f_{r}$ be a shelling order of $\sigma$. Then for $1 \leq k \leq r$, we can consider the complex
    $$ F^{k} \colon 0 \to \bigoplus_{v \not\subset \bigcup_{i=1}^{k} f_{i}} V_{v}^{3} \to \bigoplus_{e \not\subset \bigcup_{i=1}^{k} f_{i}} V_{e}^{3} \to \bigoplus_{f \not\subset \bigcup_{i=1}^{k} f_{i}} V_{f}^{3},$$
    and $F^{0} = \Ish_{\sigma}^{3}$. The sums for $v, e,$ and $f$ run through $1, 2,$ and $3$-dimensional faces, respectively. One can show that $F^{0} \supset F^{1} \supset \ldots \supset F^{r} = 0$ is a filtration by chain complexes. Hence, one can use the associated spectral sequence in order to compute the cohomologies of $\Ish_{\sigma}^{3}$.

    \begin{proof}[Proof of Theorem \ref{theo:shelling_contains_vertex}]
    We need to show that $H^2(\Ish^3_\sigma) = 0$. We have a short exact sequence of complexes
    \[ 0 \to F^{r-2} \to \Ish^3_{\sigma} \to F^0/F^{r-2} \to 0 \]
    where $F^{r-2}$ is defined to be the kernel of the map $\Ish^3_{\sigma} \to F^0/F^{r-2}$. We show that the hypothesis guarantees that $H^2(F^0/F^{r-2}) = 0$ by considering the spectral sequence associated to the shelling.
    Thus, we would be done if we could show that $H^2(F^{r-2})=0$.

    We observe that the last two faces $f_{r-1}$ and $f_{r}$ have to be adjacent to each other since $\bigcup_{j=1}^{r-2} f_{j}$ is homeomorphic to the closed disk times $\RR_{\geq 0}$. Therefore, if we denote by $e$ the two dimensional face $f_{r-1} \cap f_{r}$, then $F^{f-2}$ is exactly
    $$ V_{e}^{3} \to V_{f_{r-1}}^{3} \oplus V_{f_{r}}^{3},$$
    sitting in degrees 2 and 3. This complex is exact.
    \end{proof}

\begin{proof}[Proof of Theorem \ref{theo:lcdef-of-not-a-pyramid-dim-4}]
    We will prove that $H^2(\Ish^3_{\sigma}) \neq 0$, which would imply that $\lcdef(X) = 1$. We will use the shelling filtration to show the same.
    
    Let $f_1,\dots,f_k$ denote all the facets which contain $v_0$. Take a shelling of $\sigma$ with the first $k$ facets being $f_1,\dots,f_k$. Then $F^0/F^k$ is given by the complex
    \[ 0 \to V_{0}^{3} \to \bigoplus_{v \subset \bigcup_{i=1}^{k} f_i} V_{v}^{3} \to \bigoplus_{e \subset \bigcup_{i=1}^{k} f_i} V_{e}^{3} \to \bigoplus_{f \subset \bigcup_{i=1}^{k} f_i} V_{f}^3 \to 0. \]
    The dimensions of the 4 spaces are $4$, $3(k+1)$, $4k$ and $k$ respectively. Thus, the Euler characteristic is $1$, which implies that $H^2(F^0/F^k) \neq 0$ (since we have injectivity in degree $0$). Now, consider the short exact sequence of complexes
    \[ 0 \to F^{k} \to \Ish^3_{\sigma} \to F^0/F^k \to 0. \]
    We recall that the complex $F^{k}$ is given by
    \[ 0 \to 0 \to \bigoplus_{v \not\subset \bigcup_{i=1}^{k} f_i} V_{v}^{3} \to \bigoplus_{e \not\subset \bigcup_{i=1}^{k} f_i} V_{e}^{3} \to \bigoplus_{f \not\subset \bigcup_{i=1}^{k} f_i} V_{f}^3 \to 0. \]
    By the assumption that the rays other than $\tau_0$ span $X \otimes \mathbb{R}$, we are guaranteed that for every facet $f \not\subset \bigcup_{i=1}^{k} f_i$, there is a $2$-dimensional face $\mu \subset f$ such that $\mu \not\subset \bigcup_{i=1}^{k} f_i $. Let $f'$ denote the other facet that contains $\mu$. Now, since $V^3_\mu$ is $2$-dimensional, we can always find an $m \in V^3_\mu$ such that under the natural map in $\Ish^3_\sigma$, $m$ maps to $0 \in V^3_{f'}$ but $m$ maps to something non-zero in $V^3_f$. Doing this for all $f \not\subset \bigcup_{i=1}^{k} f_i$ guarantees that the complex $F^{k}$ is surjective at the last slot, i.e, $H^3(F^{k}) = 0$. This implies that $H^2(\Ish^3_\sigma) \neq 0$ since we showed above that $H^2(F^0/F^k) \neq 0$.
\end{proof}

    \begin{exam}\label{exam:dor}
        Here is an example of an affine toric variety $X$ of dimension $4$ such that $\lcdef(X) = 1$ and the locus where $X$ is not a rational homology manifold is 1-dimensional (i.e., the support of $\mathrm{Cone}(\QQ_X[n] \to \IC_X)$ is of dimension $1$). We will describe a $3$-dimensional rational convex polytope $P \subset \RR^3$ below. We then place the polytope in the affine hyperplane $\{x_4=1\}$ in $\RR^4$ (where the coordinates of $\RR^4$ are given by $x_1,x_2,x_3,x_4$) and take $\sigma$ to be the $4$-dimensional cone in $\RR^4$ over the polytope $P$. Finally, $X$ will be the affine toric variety associated to $\sigma$.

        In $\RR^3$, take a pyramid over an $n$-sided polygon for any $n>3$ and glue a 3-simplex along one of the triangular faces of the pyramid, while ensuring that the resulting object is a rational convex polytope. Call this polytope $P$. Consider the $4$-dimensional cone $\sigma$ over $P$ and let $X$ be the associated affine toric variety. Denote by $\tau$ the $3$-dimensional face of $\sigma$ corresponding to the $n$-sided polygon, and by $S_\tau$ the associated $1$-dimensional torus invariant subvariety. Since $\tau$ is the only non-simplicial face of $\sigma$, $S_\tau$ is precisely the locus where $X$ is non-simplicial. We can see either by Proposition \ref{prop:Euler-char-lcdef} or by Theorem \ref{theo:lcdef-of-not-a-pyramid-dim-4} that $\lcdef(X) = 1$. Additionally, the support of $\mathrm{Cone}(\QQ_X[n] \to \IC_X)$ is the locus where $X$ is non-simplicial, which is exactly $S_\tau$, a 1-dimensional subset. This shows that if we replace $\lcdef_{\mathrm{gen}}(X)$ by $\lcdef(X)$ in \cite{DOR-RHM}*{Theorem G}, the theorem fails (see \cite{DOR-RHM} for the definition of $\lcdef_{\mathrm{gen}}(X)$).
    \end{exam}

    \begin{exam} \label{exam:our-method-doesnot-work}
        We also give an example of a 4-dimensional cone $\sigma$ whose lcdef we cannot determine using our combinatorial methods. Define $\sigma$ to be spanned by the following set of 13 rays:
        \begin{align*}
            & (1,1,0,1), (1,0,1,1), (1,-1,0,1), (1,0,-1,1) \\
            & (1,1,1,0), (1,1,-1,0), (1,-1,1,0), (1,-1,-1,0) \\
            & (1,1,0,-1), (1,-1,0,-1), (1,0,-1,-1), (1,0,1,-1), (1,1,1,1).
        \end{align*}
        If we take a hyperplane section of $\sigma$, the 3-polytope we get is combinatorially equivalent to the convex hull of all midpoints of the edges of a cube, and one vertex of that cube. We can calculate (by Macaulay2 for instance) that $\dim H^{1}(\Ish_{\sigma}^{3}) = \dim H^{2}(\Ish_{\sigma}^{3}) = 1$, hence $\lcdef(\sigma) = 1$. The cone $\sigma$ has 13 1-dimensional faces, 24 2-dimensional faces and 13 3-dimensional faces. In particular, the number of 1-dimensional faces and 3-dimensional faces are equal, so Proposition \ref{prop:Euler-char-lcdef} does not apply. Additionally, every vertex of the polytope is contained in a quadrilateral or a pentagon, hence Theorem \ref{theo:lcdef-of-not-a-pyramid-dim-4} does not apply as well.
    \end{exam}
    
    {\bf Acknowledgments.} We would like to thank Mircea Musta\c{t}\u{a} for numerous helpful discussions and Lei Xue for several discussions on polytopes. We would also like to especially thank Kalle Karu for helpful suggestions.

	\bibliographystyle{alpha}
	\bibliography{Reference}

\end{document}